\documentclass[journal]{IEEEtran}

\usepackage{tikz}
\graphicspath{{images/}}
\usepackage{amsmath}
\usepackage{latexsym, amssymb}
\usepackage{graphicx}
\usepackage{amsmath, amsbsy}
\usepackage{amsopn, amstext}
\usepackage{ifpdf}
\usepackage{cancel, color}
\usepackage{epstopdf}
\usepackage{amscd}

\def\lplus{\,\rotatebox[]{-90}{$\pm$}\,}

\def\lminus{\vdash}

\def\lvtimes{\vec{\ltimes}}

\def\lvminus{\vec{\vdash}}

\def\lvplus{\vec{\lplus}}

\def\J{{\bf 1}}

\DeclareMathOperator{\Col}{Col}
\DeclareMathOperator{\Row}{Row}

\DeclareMathOperator{\argmin}{argmin}
\DeclareMathOperator{\lcm}{lcm}

\def\cal{\mathcal}

\def\pa{\partial}

\def\ra{\rightarrow}

\def\lra{\leftrightarrow}

\def\a{\alpha}
\def\b{\beta}

\def\D{\Delta}

\def\0{{\bf 0}}

\newcommand{\R}{{\mathbb R}}

\newcommand{\N}{{\mathbb N}}

\def\dsum{\mathop{\sum}\limits}

\newtheorem{thm}{Theorem}[section]
\newtheorem{dfn}[thm]{Definition}
\newtheorem{prp}[thm]{Proposition}
\newtheorem{exa}[thm]{Example}
\newtheorem{lem}[thm]{Lemma}
\newtheorem{cor}[thm]{Corollary}
\newtheorem{rem}[thm]{Remark}

\newtheorem{proof}[thm]{Proof}

\begin{document}

\title{Equivalence-Based Model of Dimension-Varying Linear Systems}

\author{Daizhan Cheng,~\IEEEmembership{Fellow,~IEEE},~~Zhenhui Xu,~~and Tielong Shen,~\IEEEmembership{Member,~IEEE}
\thanks{This work is supported partly by the National Natural Science Foundation of China (NSFC) under Grants 61733018, 61333001, and 61773371.}
\thanks{Daizhan Cheng is with the Key Laboratory of Systems and Control, Academy of Mathematics and Systems Sciences, Chinese Academy of Sciences,
		Beijing 100190, P. R. China (e-mail: dcheng@iss.ac.cn).}
\thanks{Zhenhui Xu and Tielong Shen are with Dept. of Mechanics, Sophia University, Japan  (e-mail: tetu-sin@sophia.ac.jp;~xuzhenhui@eagle.sophia.ac.jp).}
    \thanks{Corresponding author: Tielong Shen. Tel.: +81-3-32383308.}
}

\markboth{IEEE Transactions On Automatic Control, Vol. XX, No. Y, Month, 201Z}
{Cheng \MakeLowercase{\textit{et al.}}:  Dimension-Varying Linear Systems}

\maketitle

\begin{abstract}
 Dimension-varying linear systems are investigated. First, a dimension-free state space is proposed. A cross dimensional distance is constructed to glue vectors of different dimensions together to form a cross-dimensional topological space. This distance leads to projections over different dimensional Euclidean spaces and the corresponding linear systems on them, which provide a connection among linear systems with different dimensions. Based on these projections, an equivalence of vectors and an equivalence of matrices over different dimensions are proposed. It follows that the dynamics on quotient space is obtained, which provides a proper model for cross-dimensional systems. Finally, using lifts of dynamic systems on quotient space to Euclidean spaces of different dimensions, a cross-dimensional model is proposed to deal with the dynamics of dimension-varying process of linear systems. On the cross-dimensional model a control is designed to realize the transfer between models on Euclidean spaces of different dimensions.
\end{abstract}


\IEEEpeerreviewmaketitle

\section{Preliminaries}%s-1

Dimension-varying systems (DVSs) appear in various complex systems. Roughly speaking, there are two different kinds of DVSs. First kind of  DVSs have continuously time-varying dimensions. For example, (i) large-scale networks such as internet, where users are joining and withdrawing from time to time  \cite{pas04}; (ii) cellular networks, where cells are birthing and dying any time \cite{vid11}; modeling of biological systems \cite{xu81,hua95}, etc. are of this kind of DVSs.
Second kind of DVSs have ``short periods" of time-varying process. Examples are (i) docking, undocking, departure, and joining of spacecrafts \cite{yan14,jia14}, (ii) vehicle clutch system, etc. A classical way to deal with dimension-varying systems is switching \cite{yan14}. Recently, a development in optimal control of hybrid systems also involves state space and control space dimension changes, where the control design is also based on switching \cite{pak17,pak17b}. This approach ignores the dynamics of the system during the dimension-varying process. It is obviously not applicable to first kind of DVSs. Even for second kind of DVSs, the transient period may be long enough so that the dynamics during this process can not be ignored.  For instance, automobile clutch takes about $1$ second to complete a connection or separation, docking/undocking of spacecrafts takes even longer. Investigating the dynamics and designing control for transient process of dimension-varying systems can improve the performance of mechanical or other systems.

To our best knowledge, there is no proper theory or technique to model the dimension-varying dynamic process.  This paper attempts to explore the dynamic and control of dimension-varying process. It is applicable to the first kind of DVSs and the transient process of second kind of DVSs.  First of all, the dimension-free state space is introduced. A hybrid vector space structure is posed to it, and an inner product is obtained, which is then used to deduce norm and distance of the state space.  The distance makes  vectors of different dimensions (i.e., different dimensional Euclidean spaces) into a connected  topological space. As a consequence, this connection also connected linear systems with state spaces of different dimensions.

The cross-dimensional distance glues some vectors of different dimensions together (i.e., vectors with zero distance in-between), which leads to an equivalence relation on the cross-dimensional topological space. Based on this equivalence relation, a quotient space is obtained.
Linear (control) systems on quotient space is also defined and discussed. Finally, by lifting dynamic systems on quotient space to different dimensional Euclidean spaces, the dynamics for dimension-varying process of cross-dimensional systems is modeled. Then a technique is proposed to design a control to realize the required transient process.

The rest of this paper is organized as follows: Section 2 proposes a dimension free state space. First, a pseudo-vector space structure and a distance are proposed to make Euclidean spaces of different dimensions a path-wise connected topological space. Then a projection among different dimensional Euclidean spaces is discussed. Third, the vector space projection is used to deduce a projection of linear systems on different dimensional spaces. Section 3 considers cross-dynamic linear dynamic systems. First, a general model of dynamic systems on dimension free state space is discussed. Then an equivalence relation on dimension free state space is proposed. It is essentially deduced from the distance. an equivalence of matrices of different dimensions is also proposed, which is motivated by the projection of linear systems. Using these two equivalences, the corresponding quotient space is obtained, which is a standard vector space and Hausdarff topological space. Then linear (control) systems on quotient space are properly defined. In Section 4, by lifting a linear system on quotient space to Euclidean spaces of different dimensions, the dynamics of dimension-varying process of linear control systems is modeled. A technique is proposed to design required controls to realize the dimension transient process. Finally, some examples are presented to illustrate the proposed theory and related design technique.

Before ending this section we list some notations:

\begin{enumerate}

\item $\R$: Field of real numbers;
\item ${\cal M}_{m\times n}$: set of $m\times n$ dimensional real matrices.

\item $\Col(A)$ ($\Row(A)$): the set of columns (rows) of ~$A$; $\Col_i(A)$ ($\Row_i(A)$): the $i$-th column (row) of ~$A$.

\item One-entry vector: $\J_n=[\underbrace{1,\cdots,1}_n]^T$.

\item One-entry matrix: $\J_{n\times n}=[a_{i,j}]\in {\cal M}_{n\times n}$, where $a_{i,j}=1$, $\forall i,j$.

\item
$m\wedge n=gcd(m,n)$: The greatest common divisor of $m$ and $n$.

\item
$m\vee n=lcm(m,n)$: The least common multiple of $m$ and $n$.

\item
$\left<x,y\right>$, $x,y\in \R^n$: The standard  inner product on $\R^n$.

\item
$\left<x,y\right>_{{\cal V}}$, $x,y\in {\cal V}$: The dimension-free inner product on ${\cal V}$.

\item
$\|x\|$, $x\in \R^n$: The standard norm on $\R^n$.

\item
$\|x\|_{{\cal V}}$: A norm on dimension-free vector space ${\cal V}$, or an operator norm of linear operators on ${\cal V}$.

\item $\ltimes$: The first semi-tensor product (STP) of matrices.

\item $\circ$: The second semi-tensor product (STP) of matrices.

\item $\lvtimes$: The first vector product (or MV-product ) of matrix with vector.

\item $\vec{\circ}$: The second vector product of matrix with vector.

\item $\lplus$ ($\lminus$): The left M-addition (M-subtraction) of matrices.

\item $\lvplus$ ($\lvminus$): The left V-addition (V-subtraction) of vectors.

\item $x\lra y$: Vector equivalence (V-equivalence).

\item $\bar{x}$: Vector equivalence class.

\item $A\approx B$: Matrix equivalence (M-equivalence).

\item $\hat{A}$: Matrix equivalence class.

\end{enumerate}

\vskip 2mm

\section{Dimension Free State Space}

\subsection{Vector Space Structure and Distance}

Consider a dimension-varying dynamic system, its state space should be a dimension free vector space. We construct such a state space as follows:
\begin{align}\label{2.1}
{\cal V}=\bigcup_{n=1}^{\infty} {\cal V}_n,
\end{align}
where ${\cal V}_n$ is an $n$-dimensional vector space. For simplicity, we may identify ${\cal V}_n=\R^n$. A vector $x\in {\cal V}$ could be any finite dimensional vector. A dynamic system with the state $x(t)$ evolving on ${\cal V}$ is called a cross-dimensional dynamic system. As the state space of a dynamic system, ${\cal V}$ needs (i) a vector space structure; and (ii) a topological structure.

\begin{rem}\label{r2.0} Up to now, a standard way to deal with dimension-varying system is to assume that the state space is a disjoint union of finite Euclidean spaces \cite{pak17}
$$
H:= \coprod_{i=1}^k\R^{n_i},
$$
where  $\R^{n_i}$, $i=1,\cdots,k$ are considered as clopen components of $H$. So the dimension-varying can happen only by switching (or jumping) from one to another. Such a model is not available to describe the dimension-varying process.
\end{rem}

We first propose a vector space structure on ${\cal V}$:

\begin{dfn}\label{d2.1} Let $x,~y\in {\cal V}$, say, $x\in {\cal V}_m$ and  $y\in {\cal V}_n$, and $t=m\vee n$. Then an addition of $x$ and $y$, called the V-addition, is defined as follows:

\begin{align}\label{2.2}
x\lvplus y:=\left(x\otimes \J_{t/m}\right)+\left(y\otimes \J_{t/n}\right)\in {\cal V}_t.
\end{align}

Correspondingly, the V-subtraction is defined as
\begin{align}\label{2.3}
x\lvminus y:=x\lvplus (-y).
\end{align}
\end{dfn}

Recall that a set $V$ with addition and scalar product on $\R$ is a vector space if the following are satisfied: (1) $x+y=y+x$; (2)$(x+y)+z=x+(y+z)$; (3) there exists a unique $0\in V$, such that $x+0=x$, and for each $x$ there is a unique $-x$ such that $x+(-x)=0$; (4) $1\cdot x=x$; (5) $(ab)x=a(bx)$, $a,b\in \R$; (6) $(a+b)x=ax+bx$, $a,b\in \R$; (7) $a(x+y)=ax+ay$, $a\in \R$.

If only  the uniqueness of $0$, and then the uniqueness of $-x$ for each $x$, is excluded, $V$ is called a pseudo-vector space \cite{abr78}.

The following result is evident from one by one verification:

\begin{prp}\label{p2.2} ${\cal V}$ with addition, subtraction defined by (\ref{2.2}) and (\ref{2.3}) respectively, and conventional scalar product is a pseudo-vector space.
\end{prp}

\begin{dfn}\label{d2.3} Let $x,~y\in {\cal V}$, say, $x\in {\cal V}_m$ and  $y\in {\cal V}_n$, and $t=m\vee n$. The inner product of $x$ and $y$  is defined as follows:
\begin{align}\label{2.4}
\left<x\;,\;y\right>_{{\cal V}}:=\frac{1}{t}\left<x\otimes \J_{t/m}\;,\;y\otimes \J_{t/n}\right>.
\end{align}
\end{dfn}

Using this inner product, we can define a norm on ${\cal V}$.

\begin{dfn}\label{d2.4} The norm on ${\cal V}$ is defined as
\begin{align}\label{2.5}
\|x\|_{{\cal V}}:=\sqrt{\left<x\;,\;x\right>_{{\cal V}}}.
\end{align}
\end{dfn}

Finally, we define a distance on ${\cal V}$.

\begin{dfn}\label{d2.5} Let $x,~y\in {\cal V}$. The distance between $x$ and $y$ is defined as
\begin{align}\label{2.6}
d_{{\cal V}}(x,y):=\|x\lvminus y\|_{{\cal V}}.
\end{align}
\end{dfn}

\begin{rem}\label{r2.6}
\begin{itemize}
\item[(i)]~~Precisely speaking, in previous three definitions, the inner product, the norm, and the distance should be pseudo-inner product, pseudo-norm, and pseudo-distance. Because the inner product does not satisfy: $\left<x,x\right>_{{\cal V}}=0~\Rightarrow~x=0$ is unique;  correspondingly, the norm does not satisfy: $\|x\|_{{\cal V}}=0~\Rightarrow~x=0$ is unique, and the distance does not satisfy:  $d_{{\cal V}}(x,y)=0~\Rightarrow~x=y$. For statement ease, the ``pseudo-" is omitted.
\item[(ii)]~~ The metric topology deduced by the distance, denoted by ${\cal T}_d$, makes $({\cal V}, {\cal T}_d)$ a topological space  \cite{jan84}.
\end{itemize}
\end{rem}

A topological space $X$ is a Hausdorff space if for any two points $x,~y\in X$ there exist two open sets $U_x$ and $U_y$, $U_x\cap U_y=\emptyset$, such that  $x\in U_x$ and $y\in U_y$ \cite{jan84}. Unfortunately,  $({\cal V}, {\cal T}_d)$ is not a Hausdorff space.

\begin{prp}\label{p2.601}
$({\cal V}, {\cal T}_d)$ is not a Hausdorff space.
\end{prp}

\begin{proof}
Consider $x$ and $y=x\otimes \J_k$, $k>1$. Then $d_{{\cal V}}(x,y)=0$. So there is no open set $U_x$ such that $x\in U_x$ but $y\not\in U_x$.
\end{proof}

\begin{prp}\label{p2.602}
$({\cal V}, {\cal T}_d)$ is a path-wise connected topological space.
\end{prp}
\begin{proof} For any two points $x,~y\in {\cal V}$ define
a mapping $\pi:I\ra {\cal V}$ (where $I=[0,1]\subset \R$) as
\begin{align}\label{2.601}
\pi(\lambda):=\lambda x\lvplus (1-\lambda)y,\quad \lambda\in I.
\end{align}
Let $O\subset {\cal V}$ be an open set and $O\cap \pi(I)\neq\emptyset$. Then consider any $\lambda_0$ such that $\pi(\lambda_0)\in O$.
Since
$$
\begin{array}{l}
d_{{\cal V}}\left( (\lambda x\lvplus (1-\lambda)y), (\lambda_0 x\lvplus (1-\lambda_0)y)\right)\\
\leq |\lambda-\lambda_0|\left(\|x\|_{{\cal V}}+\|y\|_{{\cal V}}\right),
\end{array}
$$
as $|\lambda-\lambda_0|$ small enough, $\pi(\lambda)\in O$. That is,
$$
\left(\lambda_0-|\lambda-\lambda_0|, \lambda_0+|\lambda-\lambda_0|\right)\in \pi^{-1}(O),
$$
which implies $\pi$ is continuous. It follows that ${\cal V}$ is path-wise connected.
\end{proof}

The above proposition shows that  the distance $d_{{\cal V}}$ glues all Euclidean spaces $\R^n$, $n=1,2,\cdots$ together to form the dimension free state space ${\cal V}$. This is a key point for further discussion.

\subsection{Projection From ${\cal V}_m$ To ${\cal V}_n$}

\begin{dfn}\label{d3.1} Let $\xi\in {\cal V}_m$. The projection of $\xi$ on ${\cal V}_n$, denoted by $\pi^m_n(\xi)$, is defined as
\begin{align}\label{3.1}
\pi^m_n(\xi):=\argmin_{x\in {\cal V}_n}\|\xi\lvminus x\|_{{\cal V}}.
\end{align}
\end{dfn}

Let $m\vee n=t$ and set $\a:=t/m$, $\b:=t/n$.  Then the square error is
$$
\D:=\|\xi\lvminus x\|^2_{{\cal V}}=\frac{1}{t}\|\xi\otimes \J_{\a}-x\otimes \J_{\b}\|^2.
$$

Denote
$$
\xi\otimes \J_{\a}:=(\eta_1,\eta_2,\cdots,\eta_t)^T,
$$
where
$$
\eta_j=\xi_i,\quad (i-1)\a+1\leq j\leq i\a;\; i=1,\cdots,m.
$$

Then
\begin{align}\label{3.2}
\D=\frac{1}{t}\dsum_{i=1}^n\dsum_{j=1}^{\b}\left(\eta_{(i-1)\b+j}-x_i\right)^2
\end{align}
Setting
$$
\frac{\pa \D}{\pa x_i}=0,\quad i=1,\cdots, n
$$
yields
\begin{align}\label{3.3}
x_i=\frac{1}{\b}\left(\dsum_{j=1}^{\b} \eta_{(i-1)\b+j}\right),\quad i=1,\cdots,n.
\end{align}
That is, $\pi^m_n(\xi)=x$. Moreover, it is easy to verify that
$$\left<\xi\lvminus x, x\right>_{{\cal V}}=0.
$$
Summarizing the above argument, we have the following result.

\begin{prp} \label{p3.2} Let $\xi\in {\cal V}_m$. The projection of $\xi$ on ${\cal V}_n$, denoted by $x$, is determined by (\ref{3.3}). Moreover, $\xi\lvminus x$ is orthogonal to $x$. That is,
\begin{align}\label{3.4}
\left[\xi-\pi^m_n(\xi)\right]\perp \pi^m_n(\xi),\quad \xi\in {\cal V}_m.
\end{align}
\end{prp}

Next, we try to find the matrix expression of $\pi^m_n$, denoted by $\Pi^m_n$, such that
\begin{align}\label{3.5}
\pi^m_n(\xi)=\Pi^m_n\xi,\quad \xi\in {\cal V}_m.
\end{align}
Then, we have
$$
\begin{array}{ccl}
\eta&=&\xi\otimes \J_{\a}=\left(I_m\otimes \J_{\a}\right)\xi\\
x&=&\frac{1}{\b}\left(I_n\otimes \J_{\b}^T\right)\eta\\
~&=&\frac{1}{\b}\left(I_n\otimes \J_{\b}^T\right)\left(I_m\otimes \J_{\a}\right)\xi.
\end{array}
$$
Hence, we have
\begin{align}\label{3.6}
\Pi^m_n=\frac{1}{\b}\left(I_n\otimes \J_{\b}^T\right)\left(I_m\otimes \J_{\a}\right).
\end{align}

Using this structure, we can prove the following result.

\begin{lem}\label{l3.3}
\begin{enumerate}
\item Assume $m \geq n$, then $\Pi^m_n$ is of full row rank, and hence $\Pi^m_n(\Pi^m_n)^T$ is non-singular.
\item Assume $m \leq n$, then $\Pi^m_n$ is of full column rank, and hence $(\Pi^m_n)^T\Pi^m_n$ is non-singular.
\end{enumerate}
\end{lem}

\noindent{\it Proof.}
\begin{enumerate}
\item Assume $m\geq n$: When $m=n$, $\Pi^m_n (\Pi^m_n)^T$ is an identity matrix. So we assume $m>n$. Using the structure of $\Pi^m_n$, defined by (\ref{3.6}), it is easy to see that each row of $\Pi^m_n$ has at least two nonzero elements. Moreover, the columns of nonzero elements of $i$-th row pressed the columns of nonzero elements of $j>i$ rows, except $j=i+1$. In latter case, they may have one overlapped column. Hence,  $\Pi^m_n$ is of full row rank. It follows that  $\Pi^m_n (\Pi^m_n)^T$ is non-singular.
\item It follows from (\ref{3.6}) that
\begin{align}\label{3.7}
\Pi^n_m=\frac{\b}{\a}\left(\Pi^m_n\right)^T.
\end{align}
The conclusion is obvious.
\end{enumerate}
\hfill $\Box$

\subsection{Projection of Linear Systems}

Consider a linear system:
\begin{align}\label{4.1}
\xi(t+1)=A\xi(t),\quad \xi(t)\in \R^m.
\end{align}

Our purpose is to find a matrix $A_{\pi}\in {\cal M}_{n\times n}$, such that the
the projective system of (\ref{4.1}) on $\R^n$ is described as
\begin{align}\label{4.2}
x(t+1)=A_{\pi}x(t),\quad x(t)\in \R^n.
\end{align}
Of course, we want system (\ref{4.2}) represents the evolution of the projection $\pi(\xi(t))$. That is, the ``perfect" projective trajectory should satisfy that
\begin{align}\label{4.3}
x(t)=\pi^m_n(\xi(t)).
\end{align}
But it is, in general, not able to find such $A_{\pi}$. So we try to find a least square approximate system.

Plugging (\ref{4.3}) into (\ref{4.2}), we have
\begin{align}\label{4.4}
\Pi^m_n\xi(t+1)=A_{\pi}\Pi^m_n\xi(t).
\end{align}
Using (\ref{4.1}) and noticing that $\xi(t)$ is arbitrary, we have
\begin{align}\label{4.5}
\Pi^m_nA=A_{\pi}\Pi^m_n.
\end{align}
With the help of Lemma \ref{l3.3}, the least square solution can be obtained.

\begin{prp}\label{p4.1}
\begin{align}\label{4.6}
A_{\pi}=\begin{cases}
\Pi^m_nA(\Pi^m_n)^T\left(\Pi^m_n (\Pi^m_n)^T\right)^{-1}\quad m\geq n\\
\Pi^m_nA\left((\Pi^m_n)^T\Pi^m_n\right)^{-1}(\Pi^m_n)^T\quad m< n.
\end{cases}
\end{align}
\end{prp}

\noindent{\it Proof.} Assume $m\geq n$: Right multiplying both sides of (\ref{4.5}) by $\left(\Pi^{m}_n\right)^T$ yields the first part of (\ref{4.6}).

Assume $m<n$: We may search a solution with the following form:
$$
A_{\pi}=\tilde{A}(\Pi^m_n)^T.
$$
Then the least square solution of $\tilde{A}$ is
$$
\tilde{A}=\Pi^m_nA\left((\Pi^m_n)^T\Pi^m_n\right)^{-1}.
$$
It follows that
$$
A_{\pi}=\Pi^m_nA\left((\Pi^m_n)^T\Pi^m_n\right)^{-1}(\Pi^m_n)^T,
$$
which is the second part of (\ref{4.6}).
\hfill $\Box$

\begin{dfn}\label{d4.101} Let $A\in {\cal M}_{m\times m}$. A mapping ${\pi}^m_n: {\cal M}_{m\times m} \ra {\cal M}_{n\times n}$ is defined as
\begin{align}\label{4.601}
 {\pi}^m_n(A):=A_{\pi},
\end{align}
where $A_{\pi}$ is defined by (\ref{4.6}).
\end{dfn}

\begin{cor}\label{c4.2}
Consider a continuous linear system
\begin{align}\label{4.7}
\dot{\xi}(t)=A\xi(t),\quad \xi(t)\in \R^m.
\end{align}
Its least square approximated system is
\begin{align}\label{4.8}
\dot{x}(t)=A_{\pi}x(t),\quad x(t)\in \R^n,
\end{align}
where $A_{\pi}$ is defined by (\ref{4.6}).
\end{cor}

\noindent{\it Proof.} Since $\dot{\xi}(t)\in \R^m$, the proof is exactly the same as the one for system (\ref{4.2}).
\hfill $\Box$

Similarly, we have the following results for linear control systems.

\begin{cor}\label{c4.3}
\begin{enumerate}
\item Consider a discrete time linear control system
\begin{align}\label{4.9}
\begin{cases}
\xi(t+1)=A\xi(t)+Bu,\quad \xi(t)\in \R^m\\
y(t)=C\xi(t),\quad y(t)\in \R^p.
\end{cases}
\end{align}
Its least square approximated linear control system is
\begin{align}\label{4.10}
\begin{cases}
x(t+1)=A_{\pi}x(t)+\Pi^m_nBu,\quad x(t)\in \R^n\\
y(t)=C_{\pi}x(t),
\end{cases}
\end{align}
where $A_{\pi}$ is defined by  (\ref{4.6}), and
\begin{align}\label{4.11}
C_{\pi}=\begin{cases}
C(\Pi^m_n)^T\left(\Pi^m_n(\Pi^m_n)^T\right)^{-1},\quad m\geq n\\
C\left((\Pi^m_n)^T\Pi^m_n\right)^{-1}(\Pi^m_n)^T,\quad m< n.
\end{cases}
\end{align}

\item Consider a continuous time linear control system
\begin{align}\label{4.12}
\begin{cases}
\dot{\xi}(t)=A\xi(t)+Bu,\quad \xi(t)\in \R^m\\
y(t)=C\xi(t),\quad y(t)\in \R^p.
\end{cases}
\end{align}
Its least square approximated linear control system is
\begin{align}\label{4.13}
\begin{cases}
\dot{x}(t)=A_{\pi}x(t)+\Pi^m_nBu,\quad x(t)\in \R^n\\
y(t)=C_{\pi}x(t),\quad y(t)\in \R^p,
\end{cases}
\end{align}
where $A_{\pi}$ is defined by  (\ref{4.6}), and
$C_{\pi}$ is defined by (\ref{4.11}).
\end{enumerate}
\end{cor}

\section{Linear Systems on Quotient Space}

\subsection{Linear Systems on Dimension-Free State Space}

The trajectory of any causal dynamic system should have semi-group property. That is,
\begin{align}\label{6.1}
x(t,\tau, x(\tau,t_0,x_0))=x(t,t_0,x_0).
\end{align}
Based on this consideration, the theory of Semi-group system (briefly, S-system) has been developed as a fundamental framework for causal dynamic systems \cite{ahs87,liu08}.

\begin{dfn}\label{d6.1}
\begin{enumerate}
\item Let $G$ be a semigroup and $X$ a set. If there is an action $\varphi:G\times X\ra X$, satisfying
\begin{align}\label{6.2}
\varphi(g_1,\varphi(g_2,x))=\varphi (g_1\circ g_2,x),\quad g_1,g_2\in G, x\in X,
\end{align}
then $(G,\varphi,X)$ is called an $S_0$-system.

\item If, in addition, $G$ is a monoid (i.e., there is an identity $e\in G$), and
\begin{align}\label{6.3}
\varphi(e,x)=x,\quad \forall x\in X,
\end{align}
then $(G,\varphi,X)$ is called an $S$-system
\end{enumerate}
\end{dfn}

Consider a classical linear system
\begin{align}\label{6.4}
x_{t+1}=A(t)x(t),\;\; x(0)=x_0,\quad x(t)\in \R^n.
\end{align}
It is obviously an S-system because $A(t)\in {\cal M}_{n\times n}$. Because ${\cal M}_{n\times n}$ is a semi-group with identity $e=I_n$. Moreover, for any $x\in \R^n$ (\ref{6.2}) and (\ref{6.3}) are satisfied.

Recall our purpose, we are going to define a cross-dimensional system. To remove the dimension restriction, we set
$$
{\cal M}:=\bigcup_{m=1}^{\infty}\bigcup_{n=1}^{\infty}{\cal M}_{m,n},
$$
and consider the action of ${\cal M}$ on ${\cal V}$, which will be the model of our cross-dimensional linear systems.

To pose a semi-group structure on ${\cal M}$, we have to define a product on ${\cal M}$. The semi-tensor product (STP) is a proper product on it.

\begin{dfn}\label{d6.2} \cite{che12} The first STP of matrices is defined as follows:
\begin{align}\label{6.5}
A\ltimes B:=\left(A\otimes I_{t/n}\right)\left(B\otimes I_{t/p}\right)\in {\cal M}_{tm/n\times tq/p},
\end{align}
where $\otimes$ is Kronecker product of matrices.
\end{dfn}

It is worth noting that the STP of matrices has been proposed and investigated for near two decades, and received many applications \cite{for16,lu17,muh16}.

\begin{rem}\label{r6.201} Throughout this paper the default matrix product is the first STP. Since it is a generalization of classical matrix product,  as a convention, the symbol $\ltimes$ is mostly omitted. That is, we always assume
\begin{align}\label{6.501}
AB:=A\ltimes B.
\end{align}
Of course, when $A$ and $B$ meet the dimension requirement, i.e., classical matrix product $AB$ is defined, then (\ref{6.501}) is obviously true.
\end{rem}

In this paper for our special purpose we define another STP, called the second STP, as follows:

\begin{dfn}\label{d6.3} Let $A\in {\cal M}_{m\times n}\subset {\cal M}$ and $B\in {\cal M}_{p\times q}\subset {\cal M}$. The second STP product on ${\cal M}$ is defined as follows: Assume $t=n\vee p$, then
\begin{align}\label{6.6}
A\circ B:=\left(A\otimes J_{t/n}\right)\left(B\otimes J_{t/p}\right)\in {\cal M}_{tm/n\times tq/p},
\end{align}
where $J_k:=\frac{1}{k}{\bf 1}_{k\times k}$.
\end{dfn}

The following proposition is a key for constructing a dynamic system.

\begin{prp}\label{p6.4} $({\cal M},\circ)$ is a semi-group.
\end{prp}
\begin{proof}
It is enough to prove the associativity, that is,
\begin{align}\label{6.7}
(A\circ B)\circ C=A\circ (B\circ C),\quad A,~B,~C\in {\cal M}.
\end{align}

Let $A\in {\cal M}_{m\times n}$,  $B\in {\cal M}_{p\times q}$, $C\in {\cal M}_{r\times s}$, and denote
$$
\begin{array}{ll}
\lcm(n,p)=nn_1=pp_1,&\lcm(q,r)=qq_1=rr_1,\\
\lcm(r,qp_1)=rr_2=qp_1p_2,&\lcm(n,pq_1)=nn_2=pq_1q_2.
\end{array}
$$
Note that
$$
J_p\otimes J_q=J_{pq}.
$$
Then
$$
\begin{array}{ccl}
(A\circ B)\circ C&=&((A\otimes J_{n_1}) (B\otimes J_{p_1}))\circ C\\
~&=&(((A\otimes J_{n_1}) (B\otimes J_{p_1}))\otimes J_{p_2})(C\otimes J_{r_2})\\
~&=&(A\otimes J_{n_1p_2}) (B\otimes J_{p_1p_2})(C\otimes J_{r_2}).
\end{array}
$$
$$
\begin{array}{ccl}
A\circ (B\circ C)&=&A\circ((B\otimes J_{q_1}) (C\otimes J_{r_1}))\\
~&=&(A\otimes J_{n_2})(((B\otimes J_{q_1})(C\otimes J_{r_1}))\otimes J_{q_2})\\
~&=&(A\otimes J_{n_2}) (B\otimes J_{q_1q_2})(C\otimes J_{r_1q_2}).
\end{array}
$$

To prove (\ref{6.7}) it is enough to prove the following three equalities:

\begin{align}\label{6.8}
\begin{array}{lcr}
n_1p_2=n_2&~&(a)\\
p_1p_2=q_1q_2&~&(b)\\
r_2=r_1q_2&~&(c)
\end{array}
\end{align}

Using the associativity of least common multiple £¨or greatest common divisor) \cite{hua57}
\begin{align}\label{6.9}
\lcm(i,\lcm(j,k))=\lcm(\lcm(i,j),k),\quad i,j,k\in \N,
\end{align}
we have
\begin{align}\label{6.10}
\lcm(qn,\lcm(pq,pr))=\lcm(\lcm(qn,pq),pr).
\end{align}
Using (\ref{6.10}), we have
$$
\begin{array}{ccl}
\mbox{LHS of (\ref{6.8}) (b)}
&=&\lcm(qn,p\lcm(q,r))\\
~&=&\lcm(qn,pqq_1)\\
~&=&q\lcm(n,pq_1)\\
~&=&qpq_1q_2.
\end{array}
$$
$$
\begin{array}{ccl}
\mbox{RHS of (\ref{6.8}) (b)}
&=&\lcm(q\lcm(n,p),pr)\\
~&=&\lcm(qpp_1,pr)\\
~&=&p\lcm(qp_1,r)\\
~&=&pqp_1p_2.
\end{array}
$$
 (\ref{6.8}) (b) follows.

Using (\ref{6.8}) (b), we have
$$
\begin{array}{ccl}
n_1p_2&=&n_1\frac{q_1q_2}{p_1}=n_1\frac{q_1q_2p}{p_1p}\\
~&=&\frac{\lcm(n,p)}{n}\frac{\lcm(n,pq_1)}{pp_1}\\
~&=&\frac{\lcm(n,pq_1)}{n}=n_2.
\end{array}
$$
which proves (\ref{6.8}) (a).

Similarly,
$$
\begin{array}{ccl}
r_1q_2&=&r_1\frac{p_1p_2}{q_1}=t_1\frac{p_1p_2q}{q_1q}\\
~&=&\frac{\lcm(q,r)}{r}\frac{\lcm(r,qp_1)}{q_1q}\\
~&=&\frac{\lcm(r,qp_1)}{r}=r_2.
\end{array}
$$
which shows (\ref{6.8}) (c).
\end{proof}
%
%\begin{dfn}\label{d6.5}
%\begin{enumerate}
%\item Let $G$ be a semigroup and $X$ a set. If there is an action $\varphi:G\times X\ra X$, satisfying
%\begin{align}\label{6.8}
%\varphi(g_1,\varphi(g_2,x))=\varphi (g_1\circ g_2,x),\quad g_1,g_2\in G, x\in X,
%\end{align}
%then $(G,\varphi,X)$ is called an $S_0$-system.
%
%\item If, in addition, $G$ is a monoid (i.e., there is an identity $e\in G$), and
%\begin{align}\label{6.9}
%\varphi(e,x)=x,\quad \forall x\in X,
%\end{align}
%then $(G,\varphi,X)$ is called an $S$-system  \cite{liu08}.
%\end{enumerate}
%\end{dfn}
%
%

Our purpose is to construct an $S_0$ system $\left({\cal M},\varphi, {\cal V}\right)$.  We already know that $\left({\cal M},\circ\right)$ is a semigroup. We also need to define an action $\varphi:{\cal M}\times {\cal V}\ra {\cal V}$, which is a product of an arbitrary matrix with an arbitrary vector, called MV-product:

\begin{dfn}\label{d6.5} Let $A\in {\cal M}_{m\times n}\subset {\cal M}$ and $x\in {\cal V}_r\subset {\cal V}$. Assume $t=n\vee r$. Then the product of $A$ with $x$, called the MV-2 product, is defined as
\begin{align}\label{6.11}
A\vec{\circ} x:=\left(A\otimes J_{t/n}\right)\left(x\otimes \J_{t/r}\right).
\end{align}
\end{dfn}

\begin{prp}\label{p6.6} $\left( ({\cal M},\circ),\vec{\circ},{\cal V}\right)$ is an $S_0$-system.
\end{prp}

\begin{proof}
Let $A\in {\cal M}_{m\times n}$,  $B\in {\cal M}_{p\times q}$, $x\in {\cal V}_{r}$. We have only to prove that
\begin{align}\label{6.12}
(A\circ B)\vec{\circ} x=A\vec{\circ}(B\vec{\circ}x).
\end{align}
Mimic to the proof of Proposition \ref{p6.4}, (\ref{6.12}) can be proved.
\end{proof}

\subsection{Continuity of Generalized Linear Mapping}

In a general S- or $S_0$- system, there is no topological structure on state space $X$, and hence no continuity can be defined. But continuity is one of the most important properties of a dynamic system. Hence we need a topological structure on $X$, and then a continuity about the mapping.

\begin{dfn} \label{d6.9} Let $(G,\varphi,X)$ be an S- ($S_0$-) system.
\begin{enumerate}
\item If $X$ is a topological space and for each $g\in G$, $\varphi|_g:X\ra X$ is continuous, then $(G,\varphi,X)$  is called a weak dynamic S- ($S_0$-)system.
\item In addition, if $X$ is a Hausdorff space, then $(G,\varphi,X)$  is called a dynamic S- ($S_0$-)system.
\end{enumerate}
\end{dfn}

Recall $\left({\cal M},\vec{\circ},{\cal V}\right)$. From Section 2 we know that ${\cal V}$ is a topological space, but not Hausdorff. To show the continuity of $A\vec{\circ} x$, for a fixed $A\in {\cal M}$,  we consider the norm of $A$.

\begin{dfn}\label{d6.901} The norm of $A$, denoted by $\|A\|_{{\cal V}}$, is defined as
\begin{align}\label{6.12}
\|A\|_{{\cal V}}:=\sup_{0\neq x\in {\cal V}}\frac{\|A\vec{\circ} x\|_{{\cal V}}}{\|x\|_{{\cal V}}}.
\end{align}
\end{dfn}

First, we give two lemmas, which will be used to estimate the norm $\|A\|_{{\cal V}}$.

\begin{lem}\label{l6.902}
Assume $x\in \R^r$. Then
\begin{align}\label{6.1201}
\|x\|_{{\cal V}}=\sqrt{\frac{1}{r}}\|x\|,
\end{align}
where $\|x\|$ is the standard Euclidian norm of $x$.
\end{lem}

\begin{proof}. It is a consequence of (\ref{2.4}) and (\ref{2.5}).
\end{proof}

\begin{lem}\label{l6.903} Assume $A\in {\cal M}$. Then for any $J_r$
\begin{align}\label{6.1202}
\|A\otimes J_{r}\|=\|A\|.
\end{align}
\end{lem}

\begin{proof}. We need the following facts, which are either easily verifiable or well known facts:
\begin{itemize}
\item[(i)]~~
$$
J_r^TJ_r=J_r.
$$
\item[(ii)]~~ Denote by $\sigma(A)$ the set of eigenvalues of $A$. Then \cite{hor85}
$$
\sigma(A\otimes B)=\left\{\lambda\mu|\lambda\in \sigma(A), \mu\in \sigma(B)\right\}.
$$
It follows that
$$
\sigma_{\max}(A\otimes B)=\sigma_{\max}(A)\sigma_{\max}(B).
$$
\item[(iii)]~~ A matrix $P\in {\cal M}_{n\times n}$ is called a Markov transition matrix, if $P_{i,j}\geq 0$, $\forall i,j$, and $\dsum_{j=1}^nP_{i,j}=1$, $i=1,\cdots,n$. A markov transition matrix $P$ is a primitive matrix, if there is an integer $k\geq 1$ such that $P^k>0$ (where $P^k>0$ means $(P^k)_{i,j}>0$, $\forall i,j$ \cite{hor85}.

\item[(iv)]~~  ~$J_r$ is a primitive matrix.

\item[(v)]~~  Let $P$ be a primitive matrix. Then \cite{hor85}
$$
\sigma_{\max}(P)=1.
$$
Hence $\sigma_{\max}(J_r)=1$.
\end{itemize}

Using above facts, we have
$$
\begin{array}{ccl}
\|A\otimes J_r\|&=&\sqrt{\sigma_{\max}\left[(A^T\otimes J_r^T)(A\otimes J_r)\right]}\\
~&=&\sqrt{\sigma_{\max}\left[(A^TA)\otimes(J_r^TJ_r)\right]}\\
~&=&\sqrt{\sigma_{\max}\left[(A^TA)\right]}\\
~&=&\|A\|\\
\end{array}
$$
\end{proof}

\begin{prp}\label{p6.10} Let $A\in {\cal M}_{m\times n}$. Then
\begin{align}\label{6.13}
\|A\|_{{\cal V}}=\sqrt{\frac{n}{m}}\sqrt{\sigma_{\max}(A^TA)}.
\end{align}
\end{prp}

\begin{proof}. First, it follows from Lemma \ref{l6.902} that for $x\in {\cal V}$,
assume $x\in \R^n$, then
\begin{align}\label{6.1301}
\begin{array}{ccl}
\|A\|_{{\cal V}}&\geq& \sup_{0\neq x\in \R^n}\frac{\|A\vec{\circ} x\|_{{\cal V}}}{\|x\|_{{\cal V}}}\\
~&=& \sup_{0\neq x\in \R^n}\frac{\sqrt{\frac{1}{m}}\|Ax\|}{\sqrt{\frac{1}{n}}\|x\|}\\
~&=&\sqrt{\frac{n}{m}}\|A\|= \sqrt{\frac{n}{m}}\sqrt{\sigma_{\max}(A^TA)}.
\end{array}
\end{align}
The last equality can be found from \cite{hor85}.

On the other hand, for any $x\in {\cal V}$, say, $x\in {\cal V}_r$, then
\begin{align}\label{6.1302}
\begin{array}{ccl}
\frac{\|A\vec{\circ}x\|_{{\cal V}}}{ \|x\|_{{\cal V}}}&\leq&
\sup_{x\in {\cal V}_r} \frac{\|(A\otimes J_{t/n})(x\otimes \J_{t/r})\|_{{\cal V}}}{ \|x\otimes \J_{t/r}\|_{{\cal V}}}\\
~&\leq& \sup_{z\in {\cal V}_t} \frac{\|(A\otimes J_{t/n})z\|_{{\cal V}}}{ \|z\|_{{\cal V}}}\\
~&=& \sup_{z\in {\cal V}_t} \frac{\sqrt{\frac{n}{mt}}\|(A\otimes J_{t/n})z\|}{ \sqrt{\frac{1}{t}}\|z\|}\\
~&=&\sqrt{\frac{n}{m}}\|A\otimes J_{t/n}\|=\sqrt{{\frac{n}{m}}}\|A\|.
\end{array}
\end{align}
(\ref{6.13}) follows from (\ref{6.1301}) and (\ref{6.1302}) immediately.
\end{proof}

Using this proposition, the following result is obvious.

\begin{thm}\label{t6.11} $\left({\cal M},\vec{\circ},{\cal V}\right)$ is a weak dynamic $S_0$-system.
\end{thm}

\begin{proof}. We have only to prove the continuity. Since the topology adopted is the metric topology, the sequence continuity is enough. Let $x_n\ra x_0$. Then
$$
\|A\vec{\circ}x_n\lvminus A\vec{\circ} x_0\|_{{\cal V}} \leq \|A\|_{{\cal V}}\|x_n\lvminus x_0\|_{{\cal V}}\ra 0.
$$
\end{proof}

In fact, $\left({\cal M},\vec{\circ},{\cal V}\right)$ is a very general class of dimension-varying systems. We give an example to depict it.

\begin{exa}\label{e6.1101} Consider a constant linear system
\begin{align}\label{6.100}
x(t+1)=A\vec{\circ}x(t),
\end{align}
where
$$
A=\begin{bmatrix}
1&0&-1&0\\
0&-1&0&1
\end{bmatrix}.
$$

It is obvious that this system is quite different from the classical linear systems, because in a classical linear system the matrix $A$ must be square, and hence the trajectory remains on fixed dimensional Euclidian space. But the trajectory of this system is evolving on ${\cal V}$.

Find the trajectory for $x(0)=x_0=(1,0,1)^T$.

It is easy to calculate that
$$
x(1)=A\vec{\circ}x(0)=\frac{2}{3}(1,1,1,1,1,1)^T.
$$

Next, it is easy to see that $\R^6$ is invariant under the action of $A\vec{\circ}:=\vec{\circ}_{A}$. Moreover, when $\vec{\circ}_{A}$ is restricted on $\R^6$ it has a matrix expression as
$$
\vec{\circ}_{A}\big|_{\R^6}:=A_*,
$$
where
$$
A_*=\frac{1}{3}\begin{bmatrix}
2&1&0&-2&-1&0\\
2&1&0&-2&-1&0\\
2&1&0&-2&-1&0\\
0&-1&-2&0&1&2\\
0&-1&-2&0&1&2\\
0&-1&-2&0&1&2\\
\end{bmatrix}.
$$
Then the overall trajectory after $t=1$ is
$$
x(t+1)=(A_*)^tx(1), \quad t\geq 1.
$$
\end{exa}

Though in this paper the second STP and the MV-2 product are used to deduce the dynamic systems to meet the least square requirement, the dynamic systems constructed by first STP and MV-1 product have been discussed in \cite{chepr1}. Many properties of these two kinds of linear systems are similar.

\subsection{Quotient Vector Space}

Since ${\cal V}$ is not a standard vector space and $({\cal V}, {\cal T}_d)$ is not a Hausdorff space, it is reasonable to glue equivalent points together to form a real vector space as a Hausdorff space. To this end, we have to find proper equivalence relation.

\begin{dfn}\label{d7.4}  $x,~y\in {\cal V}$ are said to be equivalent, denoted by $x\lra y$, if there exist $\J_{\a}$ and $\J_{\b}$ such that
\begin{align}\label{7.3}
x\otimes \J_{\a}=y\otimes \J_{\b}.
\end{align}

The equivalence class is denoted by
$$
\bar{x}=\left\{y\in {\cal V}\;|\; y\lra x\right\}.
$$

The quotient space is denoted by
$$
\Omega:={\cal V}/\lra.
$$
\end{dfn}

\begin{rem}\label{r7.5}
It is necessary to verify that the relation determined by (\ref{7.3}) is an equivalence relation (i.e., it is reflexive, symmetric, and transitive). The verification is straightforward.
\end{rem}

\begin{prp}\label{p7.6} Let $x,~y\in {\cal V}$.
$d_{{\cal V}}(x,y)=0$, if and only if, $x\lra y$.
\end{prp}

\begin{proof}. Observing (\ref{2.4})-(\ref{2.6}), the conclusion follows from definitions.
\end{proof}

Next, we transfer the vector space structure from ${\cal V}$ to $\Omega$.

\begin{dfn}\label{d7.7} Let $\bar{x},\bar{y}\in \Omega$ and $a\in \R$. Then
\begin{enumerate}
\item
\begin{align}\label{7.4}
\bar{x}\lvplus \bar{y}:=\overline{x\lvplus y}.
\end{align}
\item
\begin{align}\label{7.5}
\bar{x}\lvminus \bar{y}:=\overline{x\lvminus y}.
\end{align}
\item
\begin{align}\label{7.6}
a\bar{x}:=\overline{ax}.
\end{align}
\end{enumerate}
\end{dfn}

As a corollary of Proposition \ref{p7.6}, it is ready to check the following result:

\begin{cor}\label{c7.8}
\begin{enumerate}
\item Operators defined by (\ref{7.4})-(\ref{7.6}) are properly defined.
\item $\Omega$ with addition/subtraction defined by (\ref{7.4})-(\ref{7.5}) and scalar product defined by (\ref{7.6}) is a vector space.
\end{enumerate}
\end{cor}

To get a topological structure on quotient space $\Omega$, we define the norm of $\bar{x}$ as follows:
\begin{align}\label{7.7}
\|\bar{x}\|_{{\cal V}}:= \|x\|_{{\cal V}}.
\end{align}

\begin{prp}\label{p7.9} Let $\bar{x}\in \Omega$. Then the norm of $\bar{x}$, defined by (\ref{7.7}), is well defined.
\end{prp}

\begin{proof}. Assume the smallest vector in $\bar{x}$ is $z\in {\cal V}_t$. Then any  $x\in \bar{x}$ can be expressed as $x=z\otimes \J_r$ for certain $r$. According to (\ref{2.5})-(\ref{2.6})
$$
\|z\|_{{\cal V}}=\frac{1}{\sqrt{t}}\|z\|.
$$
Now for $x$ we have
$$
\begin{array}{ccl}
\|x\|_{{\cal V}}&=&\|z\otimes \J_r\|_{{\cal V}}\\
~&=&\frac{1}{\sqrt{tr}}\|z\otimes \J_r\|\\
~&=&\frac{1}{\sqrt{tr}}\sqrt{r(x_1^2+x_2^2+\cdots+x_t^2)}\\
~&=&\frac{1}{\sqrt{t}}\|z\|\\
~&=&\|z\|_{{\cal V}}.
\end{array}
$$
That is, $\|\bar{x}\|_{{\cal V}}$ is independent of the choice of $x$. Hence, (\ref{7.7}) is properly defined.
\end{proof}

Using (\ref{7.7}), a distance can also be defined on $\Omega$ as
\begin{align}\label{7.8}
d_{{\cal V}}(\bar{x},~\bar{y}):= \|\bar{x}\lvminus \bar{y}\|_{{\cal V}}=d_{{\cal V}}(x,y).
\end{align}

Then we can also verify the following result:
\begin{cor}\label{c7.10}
\begin{enumerate}
\item The distance defined by (\ref{7.8}) is properly defined.
\item $\Omega$ with the corresponding metric topology  is a Hausdorff space.
\end{enumerate}
\end{cor}

Before ending this subsection we propose a vector equivalence for two matrices, which is used in the sequel for describing control systems on quotient space.

\begin{dfn}\label{d7.11} Let $B, ~C\in {\cal M}$. $B$ and $C$ are said to be vector equivalent, denoted by $B\lra C$, if there exist $\J_{\a}$ and $\J_{\b}$ such that
$$
B\otimes \J_{\a}=C\otimes \J_{\b}.
$$
The equivalent class of $B$ is denoted by
$$
\bar{B}:=\{C\;|\;C\lra B\}.
$$
\end{dfn}

Note that when the vector equivalence of two matrices are considered, it means both matrices are considered as sets of vectors, consisting of their columns.

\subsection{Quotient Space of Matrices}

Let $A\in {\cal M}_{m\times m}$, $m|n$, and $n=km$, where $m,n,k\in \N$. Using (\ref{4.6}), a straightforward computation shows the following result:

\begin{prp}\label{p72.1} Assume $A\in {\cal M}_{m\times m}$ and $n=km$. Then
\begin{align}\label{72.1}
{\pi}^m_n(A)=A\otimes J_{k}.
\end{align}
\end{prp}

\begin{proof}
Since $n=km$, then $\b=frac{m\vee n}{n}=1$ and $\a=frac{m\vee n}{m}=k$. Using formula (\ref{3.6}), we have
$$
\begin{array}{ccl}
\Pi^m_n&=&\frac{1}{\b}\left(I_n\otimes \J_{\b}^T\right)\left(I_m\otimes \J_{\a}\right)\\
~&=&I_n\left(I_m\otimes \J_{k}\right)=I_m\otimes \J_k.
\end{array}
$$
Plugging it into formula (\ref{4.6}) yields
$$
\begin{array}{ccl}
\pi^m_n(A)&=&\Pi^m_n A \left((\Pi^m_n)^T\Pi^m_n\right)^{-1}(\Pi^m_n)^T\\
~&=&(I_m\otimes \J_k)A(\frac{1}{k}I_m)(I_m\otimes \J_k^T)\\
~&=&\frac{1}{k}(I_m\otimes \J_k)A(I_m\otimes \J_k^T)\\
~&=&\frac{1}{k}(I_m\otimes \J_k)(A\otimes I_k)(I_m\otimes \J_k^T)\\
~&=&\frac{1}{k}(A\otimes \J_k)(I_m\otimes \J_k^T)\\
~&=&\frac{1}{k}\left(A\otimes (\J_k\J_k^T)\right)\\
~&=&A\otimes J_k.
\end{array}
$$
Note that in the above deduction $A$ was replaced by $A\otimes I_k$. This is because of Remark \ref{r6.201}.
\end{proof}

Recall
$
{\cal M}:=\bigcup_{m=1}^{\infty}\bigcup_{n=1}^{\infty}{\cal M}_{m\times n}.
$
Then a natural topology on ${\cal M}$ is defined as follows: (i) Each ${\cal M}_{m\times n}$ is a clopen subset; (ii) Within each clopen subset ${\cal M}_{m\times n}$ the Euclidean topology of $\R^{mn}$ is adopted.

Motivated by Proposition \ref{p72.1}, we propose an equivalence relation on ${\cal M}$ as follows.

\begin{dfn}\label{d72.2} Let $A,~B\in {\cal M}$. $A$ and $B$ are said to be equivalent, denoted by $A\approx B$, if there exist $J_{\a}$ and $J_{\b}$, such that
\begin{align}\label{72.2}
A\otimes J_{\a}=B\otimes J_{\b}.
\end{align}

The equivalence class is denoted by
$$
\hat{A}=\{B\;|\;B\approx A\}.
$$

The quotient space is denoted by
$$
\Xi={\cal M}/\approx.
$$
\end{dfn}

\begin{rem}\label{r72.3} It is ready to verify that (\ref{72.2}) defines an equivalence relation.
\end{rem}

Define a product on $\Xi$ as
\begin{align}\label{72.201}
\hat{A}\circ \hat{B}:= \widehat{A\circ B}.
\end{align}

Similarly to the above argument for vector case, one can verify the following easily:

\begin{prp}\label{r72.301}
\begin{enumerate}
\item (\ref{72.201}) is properly defined.
\item $\left(\Xi,~\circ\right)$ is a semi-group.
\end{enumerate}
\end{prp}

\subsection{Linear System on Quotient Space}

%\section{S-System on Quotient Space}
%

Now we are ready to define a linear system on quotient space $\Omega$. It has been proved that $\Omega$ is a vector space and topologically it is a Hausdorff space. Hence, $\Omega$ is a nice state space for investigation. A more important fact is: at $\Omega$ a point $\bar{x}$ could be the image of points in Euclidean spaces of different dimensions, hence, it is proper to describe cross-dimension dynamic systems.

We use $\Xi$ and $\Omega$ to build linear systems on quotient space.

Denote the action of  $\Xi$ on $\Omega$ as
\begin{align}\label{7.401}
\hat{A}\vec{\circ}\bar{x}:=\overline{A\vec{\circ} x}.
\end{align}

\begin{prp}\label{p7.601} The action of  $\Xi$ on $\Omega$, defined by (\ref{7.401}), is properly defined.
\end{prp}

\begin{proof} We have only to show that (\ref{7.401}) is independent of the choice of $A\in \hat{A}$ and $x\in \bar{x}$. That is, to show that if $A\approx B$ and $x\lra y$, then
\begin{align}\label{7.501}
A\vec{\circ} x \lra B\vec{\circ}y.
\end{align}

It is obvious that in equivalence class $\hat{A}$ there exists a smallest $\Lambda\in {\cal M}_{n\times p}$ such that $A=\Lambda\otimes J_s$ and $B=\Lambda\otimes J_{\a}$. Similarly, there exists $z\in {\cal V}_q$ such that $x=z\otimes \J_t$ and $y=z\otimes \J_{\b}$. Denote $\xi=p\vee q$, $\eta=ps\vee qt$, and $\eta=k\xi$. Then we have
$$
\begin{array}{ccl}
A\vec{\circ} x&=&(\Lambda\otimes J_s)\vec{\circ} (z\otimes \J_t)\\
~&=&\left( \Lambda\otimes J_s\otimes J_{\eta/ps} \right)\left(z\otimes \J_t\otimes \J_{\eta/qt}\right)\\
~&=&\left( \Lambda\otimes J_{\xi/p}\otimes J_{k} \right)\left(z\otimes \J_{\xi/q}\otimes \J_{k}\right)\\
~&=&\left[\left( \Lambda\otimes J_{\xi/p}\right)\left(z\otimes \J_{\xi/q}\right)\right] \otimes \left(J_k\J_{k}\right)\\
~&~&=\left(\Lambda \vec{\circ} z\right)\otimes \J_{k}.
\end{array}
$$
Hence
$$
A\vec{\circ} x\lra \Lambda \vec{\circ} z.
$$
Similarly, we have
$$
B\vec{\circ} y\lra \Lambda \vec{\circ} z.
$$
(\ref{7.501}) follows.
\end{proof}

Now it is clear that $(\Xi,\vec{\circ},\Omega)$ is an $S_0$ system. Expressing it in classical form yields

\begin{itemize}
\item[(i)]~~ Discrete time linear system:
\begin{align}\label{7.502}
\bar{x}(t+1)=\hat{A}\vec{\circ} \bar{x}(t).
\end{align}
\item[(ii)]~~ Continuous time linear system:
\begin{align}\label{7.503}
\dot{\bar{x}}(t)=\hat{A}\vec{\circ} \bar{x}(t).
\end{align}
\end{itemize}

To prove such a system is a dynamic system, we have to show that for a given $\hat{A}$ the mapping $\bar{x}\mapsto \hat{A}\vec{\circ} \bar{x}$ is
continuous. To this end, we define the norm of $\hat{A}$. The following definition is natural.

\begin{dfn}\label{d7.602} Assume $\hat{A}\in \Xi$. Its norm is defined as
\begin{align}\label{7.801}
\|\hat{A}\|_{{\cal V}}:= \|A\|_{{\cal V}}.
\end{align}
\end{dfn}

\begin{prp}\label{p7.801} Let $\hat{A}\in \Xi$. Then the norm of $\hat{A}$, defined by (\ref{7.801}), is well defined.
\end{prp}

\begin{proof}
Assume $\Lambda\in \hat{A}$ is the smallest element of the class. Then each $A\in \hat{A}$ can be expressed as
$A=\Lambda\otimes J_r$.

Using Lemma \ref{l6.903} and Proposition \ref{p6.10}, for $A\in{\cal M}_{m\times n}$ and any $J_s$, we have
$$
\|A\otimes J_s\|_{{\cal V}}=\sqrt{\frac{n}{m}}\|A\otimes J_s\|=\sqrt{\frac{n}{m}}\|A\|=\|A\|_{{\cal V}}.
$$
Hence, we can get
$$
\|A\|_{{\cal V}}=\|\Lambda\otimes J_r\|_{{\cal V}}= \|\Lambda\|_{\cal{V}},
$$
which is independent of the particular choice of $A$.
\end{proof}

Then we have the following  result:

\begin{cor}\label{c7.901}
The discrete time $S_0$-system (\ref{7.502}) or continuous time $S_0$-system (\ref{7.503}) on quotient space $\Omega$ are properly defined dynamic systems.
\end{cor}

\section{Dynamics of Dimension-Varying Process}

Though the cross-dimensional systems discussed in previous sections are very general, this section is particulary focused on the transient dynamics of systems, which has classical fixed dimensions during normal time, and only on dimension transient period the system changes its model from one to another, which have different dimensions. This kind of dimension-varying systems are practically important.

\subsection{Modeling Transient dynamics via Equivalent Dynamic Systems}

\begin{dfn}\label{d9.1}
\begin{enumerate}
\item Assume a (standard) discrete time linear control system
\begin{align}\label{79.8}
\begin{array}{l}
x(t+1)=A(t) x(t)+B(t)u(t),\quad u(t)\in \R^m\\
y(t)=H(t) x(t),\quad y(t)\in \R^p,~~x(t)\in \R^n,
\end{array}
\end{align}
is given. The following system on quotient space $\Omega$ is called the projecting system of (\ref{79.8}):
\begin{align}\label{79.10}
\begin{array}{l}
\bar{x}(t+1)=\hat{A}(t)\vec{\circ} \bar{x}(t)+\bar{B}(t)u(t)\\
y(t)=\hat{H}(t)\vec{\circ} \bar{x}(t),\quad \bar{x}(t)\in \Omega.
\end{array}
\end{align}
\item Assume a (standard) continuous time linear control system
\begin{align}\label{79.9}
\begin{array}{l}
\dot{x}=A(t) x(t)+B(t)u(t),\quad x\in \R^n, ~u(t)\in \R^m\\
y(t)=H(t) x(t),\quad y(t)\in \R^p
\end{array}
\end{align}
is given. The following system on quotient space $\Omega$ is called the projecting system of (\ref{79.9}):
\begin{align}\label{79.11}
\begin{array}{l}
\dot{\bar{x}}(t)=\hat{A}(t)\vec{\circ} \bar{x}(t)+\bar{B}(t)u(t)\\
y(t)=\hat{H}(t)\vec{\circ} \bar{x}(t),\quad \bar{x}(t)\in \Omega.
\end{array}
\end{align}
\item Assume a discrete time linear control system on quotient space $\Omega$ as (\ref{79.10}) is given. System (\ref{79.8}) is called its lifting system on $\R^n$, if $A\in \hat{A}$, $B\in \bar{B}$, and $H\in \hat{H}$.
 \item Assume a continuous time linear control system on quotient space $\Omega$ as (\ref{79.11}) is given. System (\ref{79.9}) is called its lifting system on $\R^n$, if $A\in \hat{A}$, $B\in \bar{B}$, and $H\in \hat{H}$.
\end{enumerate}

\end{dfn}

Since a system on quotient space is a set of equivalent systems with various dimensions, dimension-varying is not a problem for such a system. Then the transient dynamics can be considered as a dynamic process on quotient space. This is our main idea for dealing with transient dynamics.

\begin{dfn}\label{d9.2} Let $\Theta_0$ be a linear control system on quotient space. $\Theta_n$ be its lifting on $\R^n$. Then all such lifting systems are said to be equivalent.
\end{dfn}

It follows from definition that
\begin{prp}\label{p9.3} Linear control systems $(A,B,C)$ and $(A',B',C')$ are equivalent, if and only if, there exist $r,~s\in \N$, such that
\begin{align}\label{9.1}
\begin{cases}
A\otimes J_r=A'\otimes J_s\\
B\otimes \J_r=B'\otimes \J_s\\
C\otimes J_r=C'\otimes J_s.
\end{cases}
\end{align}
\end{prp}

Consider a dimension-varying system. Without loss of generality, we assume it is evolving from a model $\Sigma_1$ to another model $\Sigma_2$ over a transient period, where

\begin{align}\label{5.1}
\Sigma_1:~\dot{x}(t)=Ax(t)+Bu(t),\quad x\in \R^p;
\end{align}
and
\begin{align}\label{5.2}
\Sigma_2:~\dot{y}(t)=Ey(t)+Fu(t),\quad y\in \R^q.
\end{align}

We consider the transient dynamics of the system from starting time $t=t_0$ to ending time $t=t_e>t_0$.

To keep the dynamics of dimension-varying process a linear model, we assume the dynamics is a linear combination of the two models. That is,

\begin{itemize}
\item Assumption A1: The dynamics of dimension-varying process is
That is:
\begin{align}\label{5.3}
\begin{array}{l}
\dot{z}(t)=\mu \dot{x}(t)\lvplus (1-\mu)\dot{y}(t),\\
z(t_0)=x(t_0),\; z(t_e)=y(t_e).
\end{array}
\end{align}
\end{itemize}

We have two ways to choose $\mu$:
\begin{itemize}
\item[(i)]~~ Constant Parameter:

Choose $0 <\mu<1$ being constant, which leads to a constant linear system.
\item[(ii)]~~ Time-varying Parameter:

Choose $0\leq \mu(t)\leq 1$, and
$$
\mu(t)=\begin{cases}
1,\quad t=t_0\\
0,\quad t=t_e.
\end{cases}
$$
\end{itemize}

Let $n=p\vee q$ be the least common multiple of $p$ and $q$. Using (\ref{4.6}), we can project $\Sigma_1$ into $\R^n$ as
\begin{align}\label{5.4}
\dot{z}(t)=A_1z+B_1u,
\end{align}
where
$$
A_1=\Pi^p_nA\left((\Pi^p_n)^T\Pi^p_n\right)^{-1}(\Pi^p_n)^T,
$$
$$
B_1=\Pi^p_nB.
$$
Similarly, projecting $\Sigma_2$ into $\R^n$ yields
\begin{align}\label{5.5}
\dot{z}(t)=A_2z+B_2v,
\end{align}
where
$$
A_2=\Pi^q_nE\left((\Pi^q_n)^T\Pi^q_n\right)^{-1}(\Pi^q_n)^T,
$$
$$
B_2=\Pi^q_nF.
$$
According to (\ref{5.3}), the transient dynamics becomes
\begin{align}\label{5.6}
\dot{z}(t)=\left[\mu A_1+(1-\mu)A_2\right]z+\mu B_1u+(1-\mu) B_2v.
\end{align}

\begin{dfn}\label{d5.1} A dimension transience is properly realized if there exist $u(t)$ and $v(t)$ such that, stating from $z_0=x_0\otimes {\bf 1}_{n/p}$, the ending state of (\ref{5.6}) satisfies
\begin{align}\label{5.7}
z(t_e)=y(t_e)\otimes {\bf 1}_{n/q}\in \R^q\otimes {\bf 1}_{n/q}.
\end{align}
\end{dfn}

\begin{rem}\label{r5.2}
\begin{enumerate}
\item If the constant parameter is assumed, the parameter $\mu$ is determined by the system model. Assume $m_1$ and $m_2$ are ``formal masses"  of the two systems, then using the law of conservation of momentum, we have $\mu=\frac{m_1}{m_1+m_2}$.

\item  If the time-varying parameter is assumed, the easiest way is to assume the parameter is a linear function. That is
$$
\mu(t)=\frac{(t_e-t_0)-(t-t_0)}{t_e-t_0}.
$$
\end{enumerate}
\end{rem}

The following result is easily verifiable.

\begin{prp}\label{p5.3} A dimension transience is properly realized if $x(t_0)\otimes {\bf 1}_{n/p}$ is controllable to a point of $\R^q\otimes {\bf 1}_{n/q}$.
\end{prp}

\subsection{Illustrative Examples}

In this section two examples are presented to illustrate our results. In first example constant parameter is assumed. In second example time-varying parameter is used.

\begin{exa}\label{e5.4} Consider a dimension-varying system, which has two models as
\begin{equation}\label{sys1}
\Sigma_1:\left\{\begin{split}
\dot{x}_1 &= x_2\\
\dot{x}_2 &= u;
\end{split}
\right.
\end{equation}
\begin{equation}\label{sys2}
\Sigma_2:\left\{\begin{split}
\dot{y}_1 &= y_3\\
\dot{y}_2 &= v\\
\dot{y}_3 &= y_2.
\end{split}
\right.
\end{equation}

Assume during the period $[0,10]$ seconds, the system runs in $\Sigma_1$, whereas at the tenth second, the system changes and evolves in the transient dynamics. Then, after one second, the system arrives at $\Sigma_2$. The initial time and the end time of the transient dynamics are denoted as $t_0=10$ and $t_e=11$ respectively. Let $x(0)=(0,0)^\mathrm{T}$, $x(t_0)=(1,-1)^{\mathrm{T}}$, $y(t_0) = (0,1,1)^\mathrm{T}$, $\mu=0.5$ (\sl i.e, $m_1=m_2$).

Here we have $p=2$ and $q=3$, hence $n=p\vee q=6$. Using (\ref{3.6}) and (\ref{4.6}), the projective systems of $\Sigma_1$ and $\Sigma_2$, denoted by $\Sigma^{\pi}_1$ and $\Sigma^{\pi}_2$, respectively, are
$$
\dot{z}=A^{\pi}_1z+B^{\pi}_1u;
$$
and
$$
\dot{z}=A^{\pi}_2z+B^{\pi}_2v,
$$
where
$$
\begin{array}{ccl}
A^{\pi}_1&=&\Pi^2_6A_1\left[(\Pi^2_6)^T(\Pi^2_6)\right]^{-1}(\Pi^2_6)^T\\
~&=&\frac{1}{3}\begin{bmatrix}
0&0&0&1&1&1\\
0&0&0&1&1&1\\
0&0&0&1&1&1\\
0&0&0&0&0&0\\
0&0&0&0&0&0\\
0&0&0&0&0&0\\
\end{bmatrix};
\end{array}
$$
$$
B^{\pi}_1=\Pi^2_6B_1=[0,0,0,1,1,1]^T.
$$
$$
\begin{array}{ccl}
A^{\pi}_2&=&\Pi^3_6A_2\left[(\Pi^3_6)^T(\Pi^3_6)\right]^{-1}(\Pi^3_6)^T\\
~&=&\frac{1}{2}\begin{bmatrix}
0&0&0&0&1&1\\
0&0&0&0&1&1\\
0&0&0&0&0&0\\
0&0&0&0&0&0\\
0&0&1&1&0&0\\
0&0&1&1&0&0\\
\end{bmatrix};
\end{array}
$$
$$
B^{\pi}_2=\Pi^2_6B_2=[0,0,1,1,0,0]^T.
$$
Then the transient dynamics becomes
\begin{align}\label{5.8}
\dot{z}=A^*z+B^*_1u+B^*_2v,
\end{align}
where
$$
\begin{array}{ccl}
A^*&=&\mu A^{\pi}_1+(1-\mu) A^{\pi}_2\\
~&=&\begin{bmatrix}
0 & 0 & 0 & 1/6 & 5/12 & 5/12 \\
0 & 0 & 0 & 1/6 & 5/12 & 5/12 \\
0 & 0 & 0 & 1/6 & 1/6 & 1/6 \\
0 & 0 & 0 & 0 & 0 & 0 \\
0 & 0 & 1/4 & 1/4 & 0 & 0 \\
0 & 0 & 1/4 & 1/4 & 0 & 0 \\
\end{bmatrix}.
\end{array}
$$
$$
\begin{array}{l}
B^*_1=\mu B^{\pi}_1=[0,0,0,1/2,1/2,1/2]^T\\
B^*_2=(1-\mu) B^{\pi}_2=[0,0,1/2,1/2,0,0]^T.
\end{array}
$$

\begin{align}\label{5.9}
z(t_0)=\Pi^2_6 x(t_0).
\end{align}

When $t\in[0,t_0]$, we choose a PD controller ($K_p=10$, $K_d=5$) to control system (\ref{sys1}) to reach $x(t_0)=(1,-1)$. Then, during $[t_0,t_e]$, to verify whether the dimension transience can be properly realized, we may choose
$$
z(t_e)=[1,1,2,2,1,1]^T\in \R^3\otimes {\bf 1}_2
$$
to see if the system (\ref{5.8}) is controllable from $z(t_0)$ to $z(t_e)$. Then we deign an open-loop control law for the transient system. When $t\in[t_e,25]$, we design a state-feedback controller to stabilize the system (\ref{sys2}). The time response of the system according to the three period, $[0,10]$, $[10,11]$, and $[11,25]$, are shown in Fig. \ref{fig1}, Fig. \ref{fig2}, and Fig. \ref{fig3}, respectively. Furthermore, the whole trajectory in the state space with three, actually from two-dimension to the three-dimension, is as shown in Fig. \ref{fig4}, where the dashed line represents the projective system of the transient system (\ref{5.8}) in $\mathbb{R}^3$. The time response of the projective system of the system (\ref{5.8}) is shown in Fig. \ref{fig5}. It should be noted that the trajectory during the transient period is re-coordinated as shown in the note due to the large scale.

\begin{figure}[!htb]
  \centering
  \includegraphics[width=9.5cm,height=5.5cm]{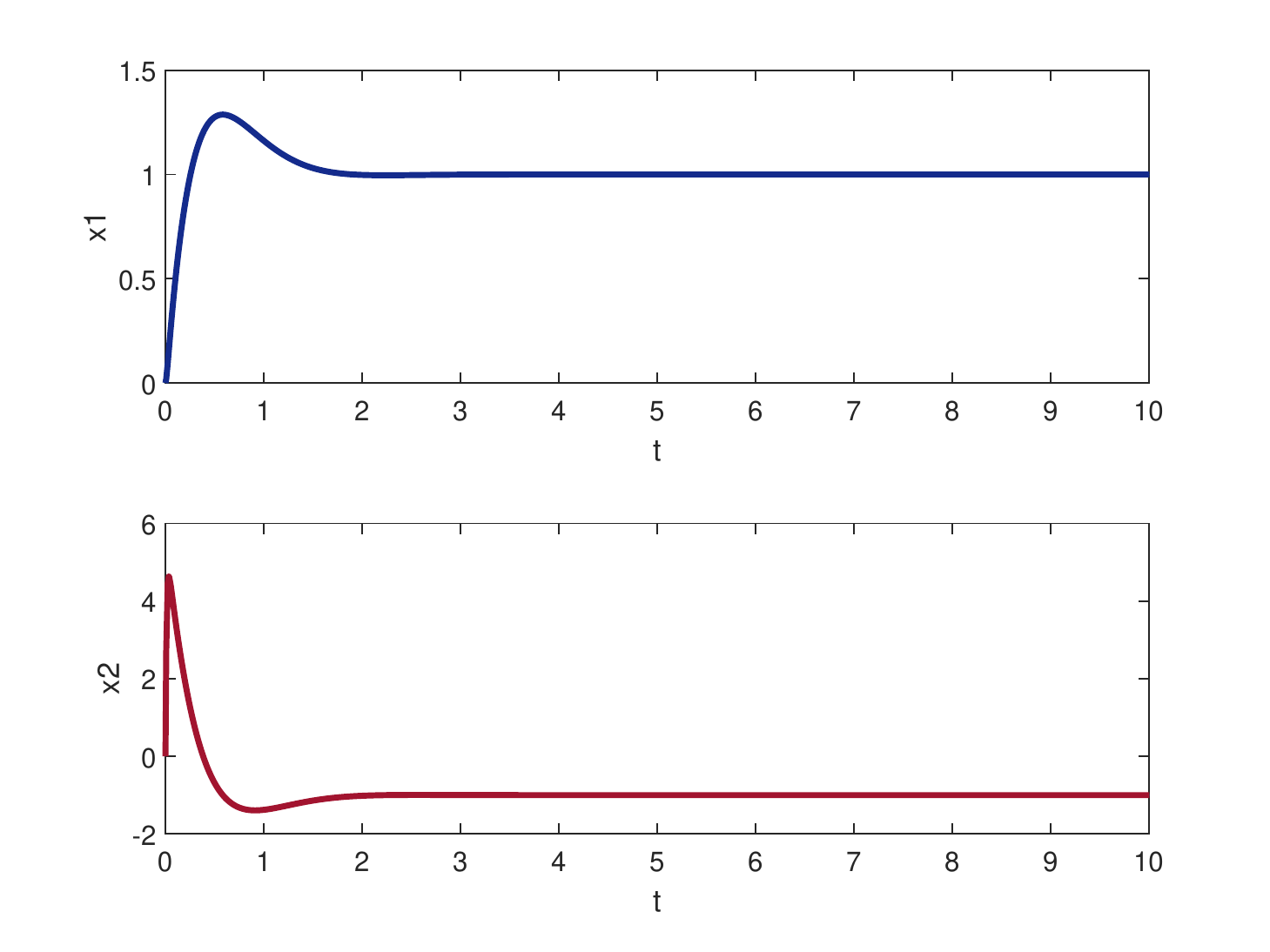}
  \caption{ The profile of states of the system (\ref{sys1}) }\label{fig1}
\end{figure}
\begin{figure}[!htb]
  \centering
  \includegraphics[width=9.6cm,height=5.5cm]{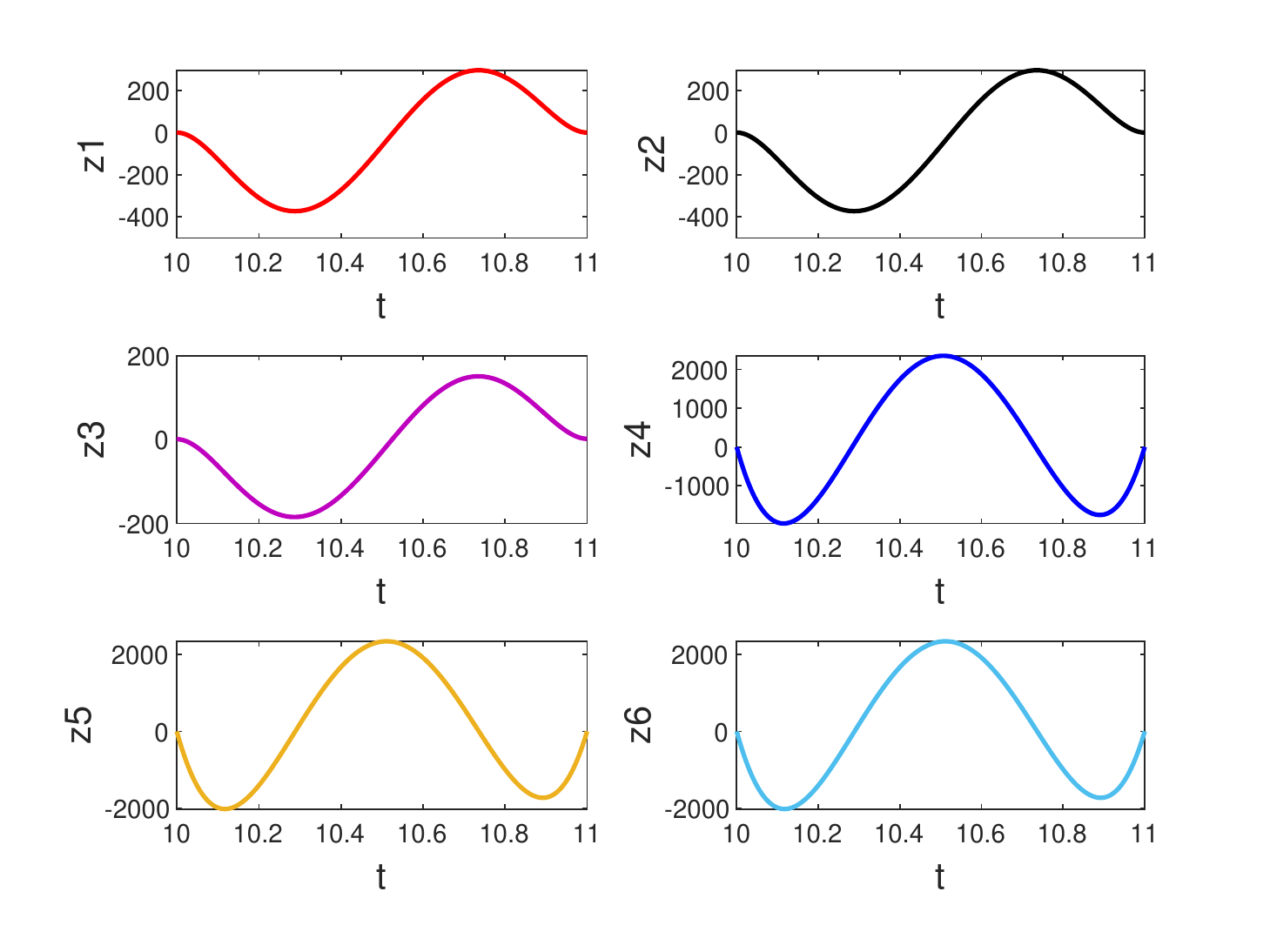}
  \caption{The profile of states of system (\ref{5.8})}\label{fig2}
\end{figure}

\begin{figure}[!htb]
  \centering
  \includegraphics[width=9.6cm,height=5.5cm]{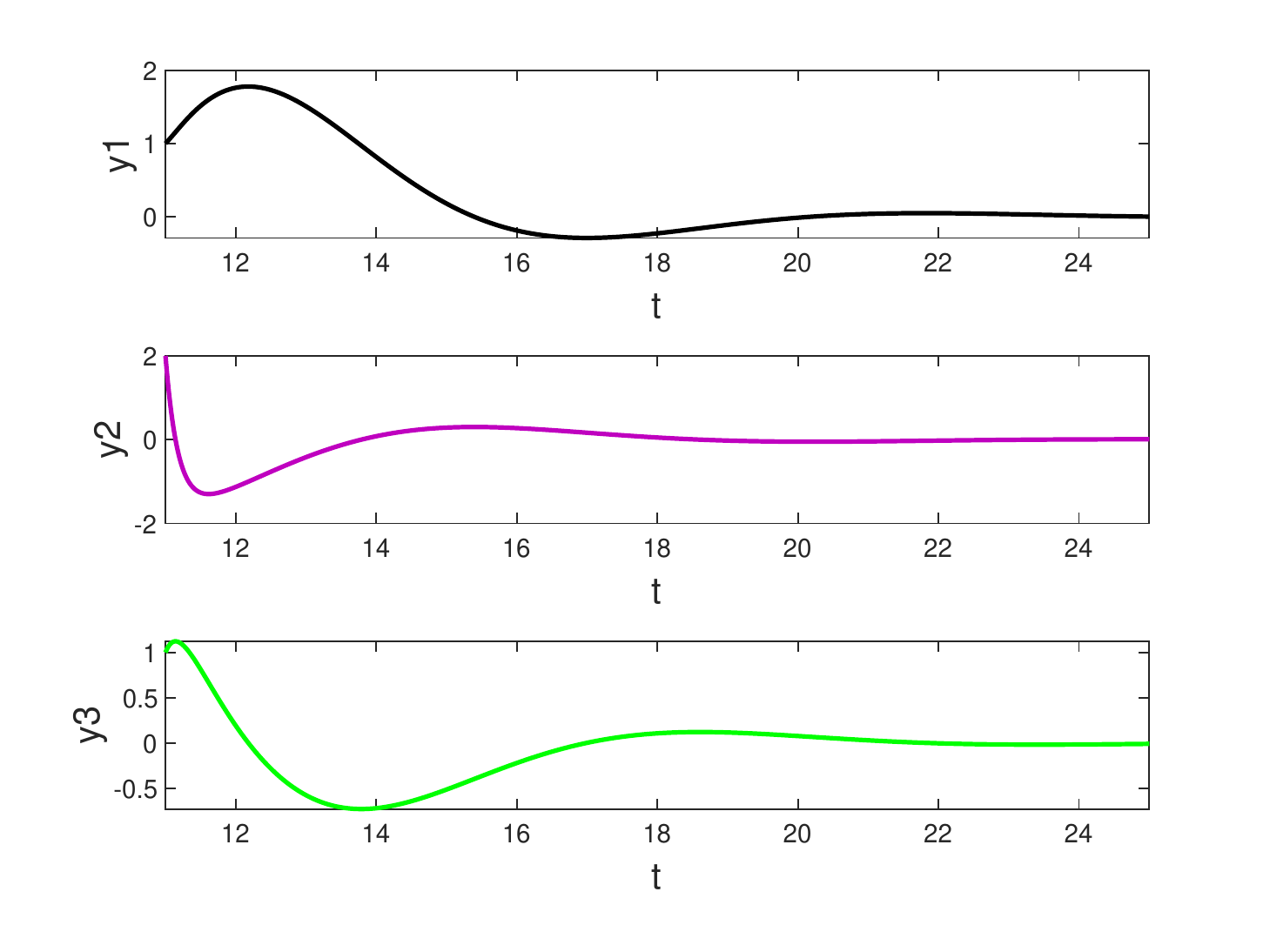}
  \caption{ The profile of states of the system (\ref{sys2}) }\label{fig3}
\end{figure}

\begin{figure}[t]
  \centering
  \includegraphics[width=9.5cm,height=5.5cm]{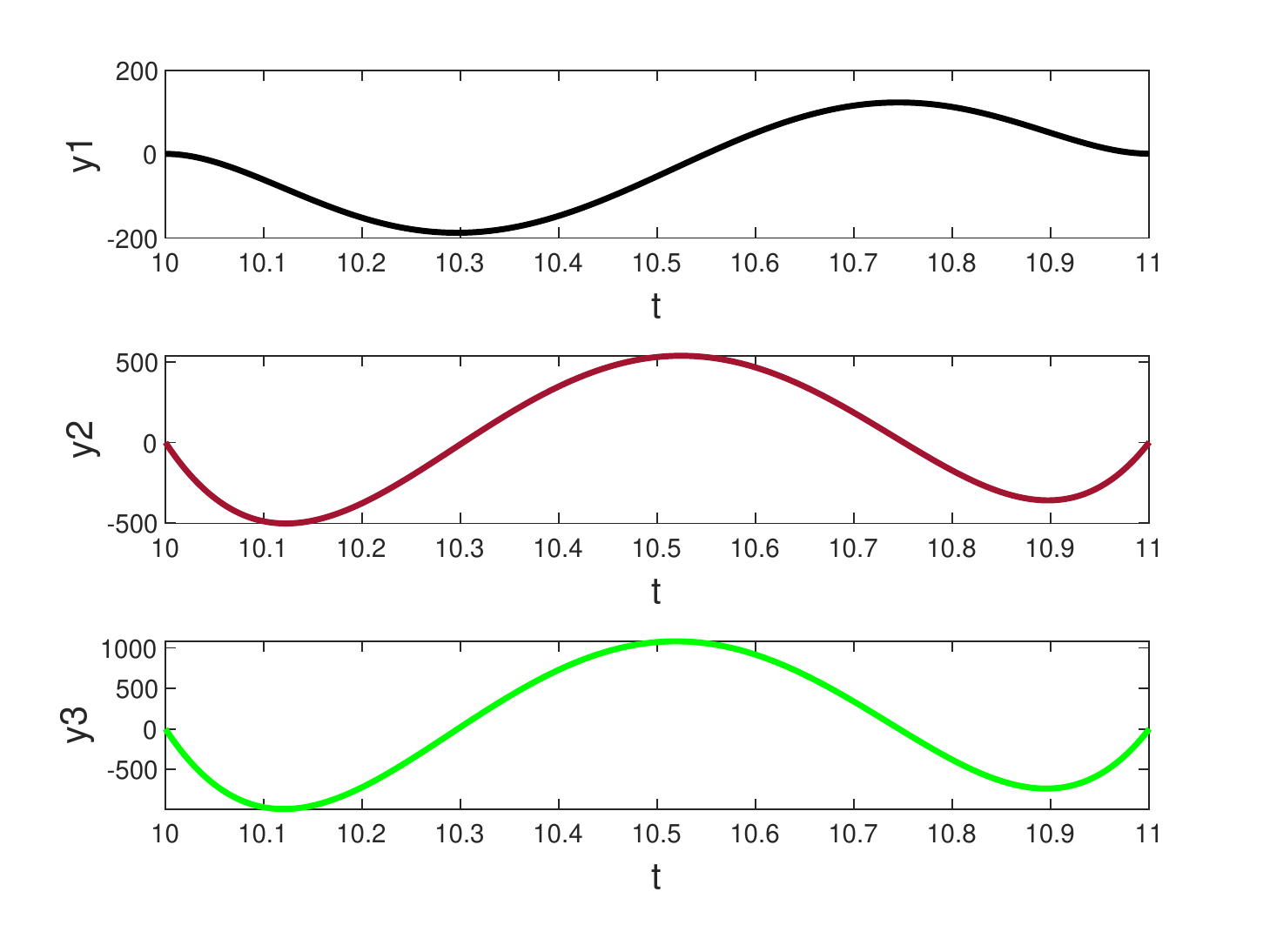}
  \caption{ The profile of states of the projective system of (\ref{5.8}) in $\mathbb{R}^3$ }\label{fig5}
\end{figure}

\begin{figure}[!htb]
  \centering
  \includegraphics[scale=0.32]{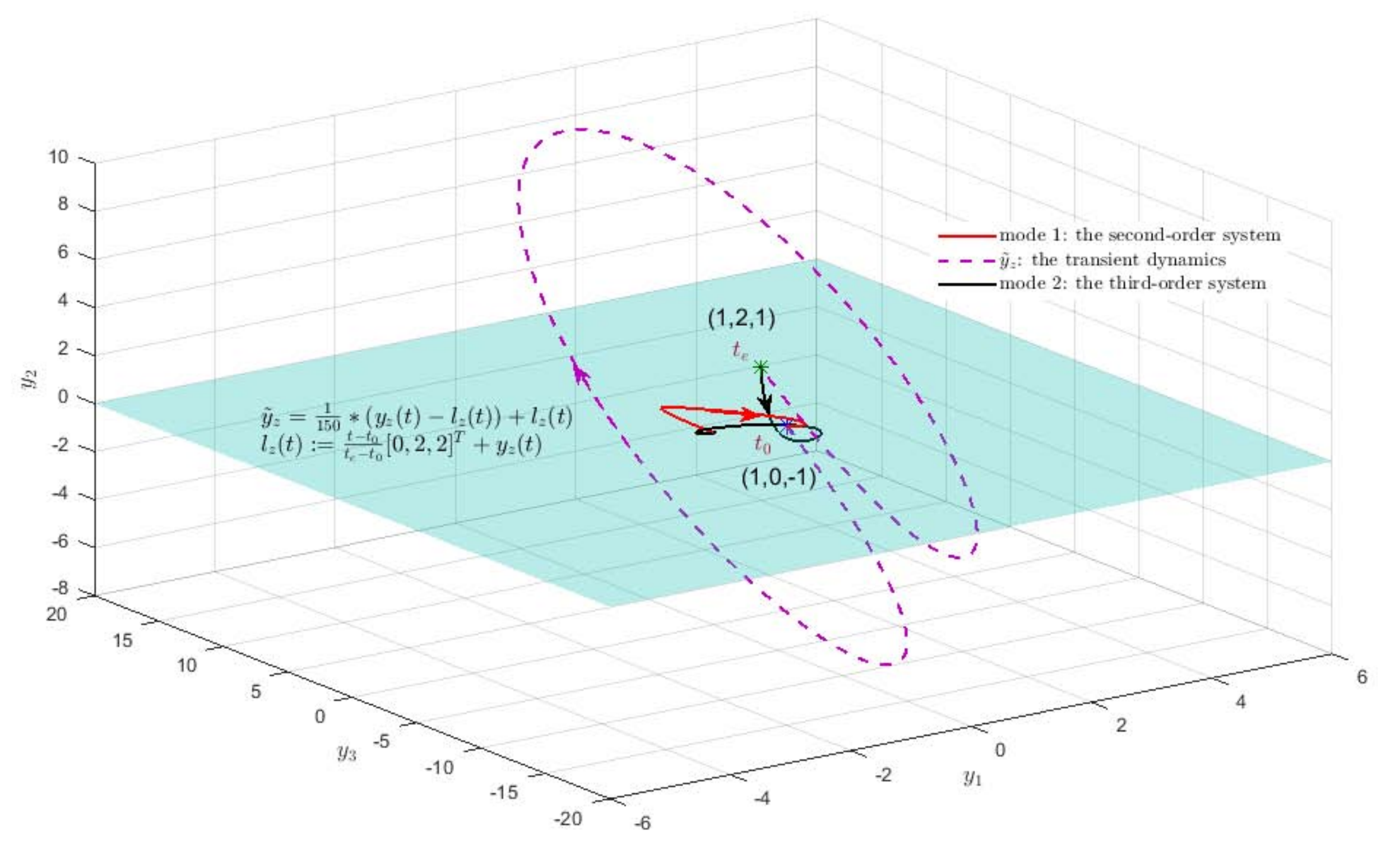}
 \caption{ The trajectory of the cross-dimension system }\label{fig4}
\end{figure}

\end{exa}

\begin{exa}\label{e5.5} The clutch is a typical device widely used in automotive engineering. Dynamics of a clutch system is illustrated in Figure. \ref{clutch}, where only two inertia system is sketched for sake of simplicity. In Figure. \ref{clutch}, the left side of the clutch is connected to the power source such as combustion engine, electric motor, etc. And the right side is connected to the load, usually the input axis of the transmission box connected to the differential gear and the wheel along the powertrain. Obviously, when the clutch is disengaged the motion of the system consists of two rotational mass, and the motion of the two inertia does not couple each other. The dynamics can be represented with two decoupled rotational dynamics:
\begin{equation}\label{e2sys1}
\Sigma_1:\left\{\begin{split}
J_i\dot{\omega}_i &= -d_i\omega_i+\tau_i \\
J_o\dot{\omega}_o &= -d_o\omega_o-\tau_o,
\end{split}
\right.
\end{equation}
where $\omega_i$ and $\omega_o$ denote the rotational speed of the axis, $\tau_i$ and $\tau_o$ denote the active torque generated by the power source and the load torque which is reacting torque to force the load, respectively. $d_i$ and $d_o$ denote the friction coefficient of the corresponding axis.

On the other hand, when the clutch is engaged the two axis connected rigidly and rotational motion becomes one inertia and one dimensional dynamics.

\begin{equation}\label{e2sys2}
\Sigma_2:(J_i+J_o)\dot{\omega}_o =-(d_i+d_o)\omega_o+\tau_i-\tau_o,
\end{equation}
and $\omega_i = \omega_o$.

\begin{figure}[!htb]
  \centering
  \includegraphics[scale=0.5]{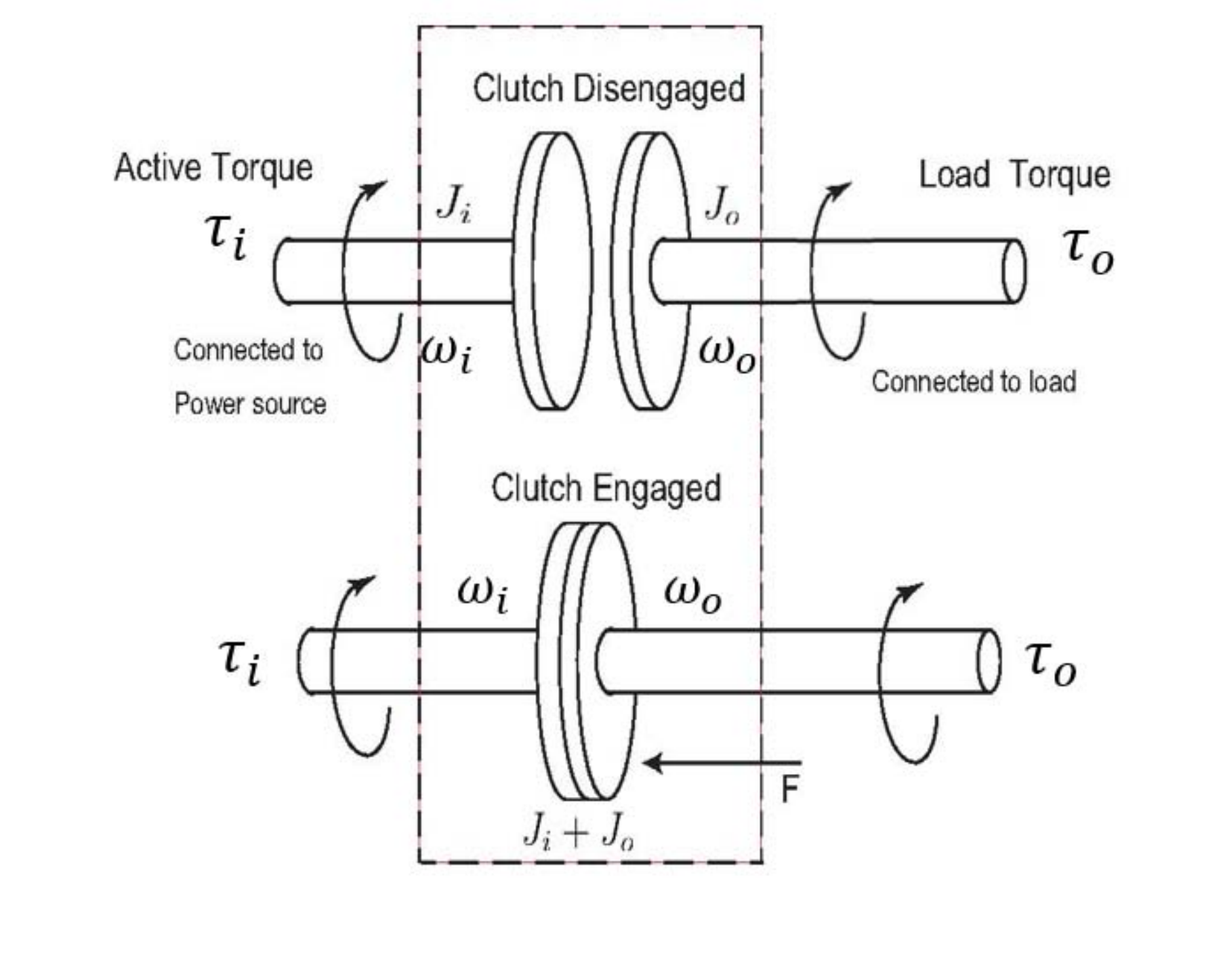}
  \caption{The clutch system}\label{clutch}
\end{figure}

It means that the clutch system can be described as two dimensional system or one dimensional system according to the state of the clutch. If the state of clutch is ''disengaged'', the system is two dimension, and the state of ''engaged'' leads to one dimension. The transient process of this system from ''disengaged'' to ''engaged'' is conducted by adjusting the interacting torque $\tau_c$ between the two inertia, {\sl i.e.}, during the transient process the total torque acts on $J_i$ is $\tau_i-\tau_c$, and on $J_o$ is $\tau_c-\tau_o$, respectively. Adjusting these total torques by $\tau_c$ will complete the transient from two mass to one combined mass system. In automotive practice, it is physically implemented by external force $F$ that acts on the clutch disc since the interacting torque is generated as $\tau_c=F c R_a\psi(\omega_i,\omega_o)$ by this operation, in which $c$ is the friction coefficient of the clutch surface material, $R_a$ is the active radius of the clutch plates, and $\psi(\omega_i,\omega_o)$ is a nonlinear function, see \cite{ser04,tem018}. Usually, the transient process is requested to be finished quickly, less than $0.8\sim1.0$ sec. Equivalently, this clutch operating process is nothing but adjusting the total torque acting on the two inertia to get complete synchronized speed for connecting rigidly.

This process can be described with the proposed model of varying dimensional system. Here we have $p=2$ and $q=1$, hence $n=p\vee q=2$. Using (\ref{3.6}) and (\ref{4.6}), the projective systems of $\Sigma_1$ and $\Sigma_2$, denoted by $\Sigma^{\pi}_1$ and $\Sigma^{\pi}_2$, respectively, are
$$
\dot{z}=A^{\pi}_1z+B^{\pi}_1u;
$$
and
$$
\dot{z}=A^{\pi}_2z+B^{\pi}_2u,
$$
where
$$
\begin{array}{ccl}
A^{\pi}_1=\begin{bmatrix}
-\frac{d_i}{J_i}&0\\
0&-\frac{d_o}{J_o}\\
\end{bmatrix};
\end{array}
$$
$$
\begin{array}{ccl}
B^{\pi}_1=\begin{bmatrix}
\frac{1}{J_i}&0\\
0&-\frac{1}{J_o}\\
\end{bmatrix};
\end{array}
$$
$$
\begin{array}{ccl}
A^{\pi}_2&=&\Pi^1_2A_2\left[(\Pi^1_2)^T(\Pi^1_2)\right]^{-1}(\Pi^1_2)^T\\
~&=&\frac{1}{2}\begin{bmatrix}
-\frac{d_i+d_o}{J_i+J_o}&-\frac{d_i+d_o}{J_i+J_o}\\
-\frac{d_i+d_o}{J_i+J_o}&-\frac{d_i+d_o}{J_i+J_o}\\
\end{bmatrix};
\end{array}
$$
$$
\begin{array}{ccl}
B^{\pi}_2&=&\Pi^1_2B_2\\
~&=&\begin{bmatrix}
\frac{1}{J_i+J_o}&-\frac{1}{J_i+J_o}\\
\frac{1}{J_i+J_o}&-\frac{1}{J_i+J_o}\\
\end{bmatrix};
\end{array}
$$
and
\begin{align*}
u=\left[\begin{array}{c}
                  \tau_i \\
                  \tau_o
                \end{array}
                \right].
\end{align*}

Then the transient process from $\Sigma_1$ to $\Sigma_2$ can be represented as
\begin{align}\label{e25.8}
\Sigma^*: ~\dot{z}=A^*z+B^*u,
\end{align}
by defining
\begin{equation*}
\begin{split}
A^*=&(1-\mu)\begin{array}{cc}
\left[
               \begin{array}{cc}
                 -\frac{d_i}{J_i} & 0 \\
                 0 & -\frac{d_o}{J_o} \\
               \end{array}
             \right]\end{array}+\frac{\mu}{2}\left[\begin{array}{cc}
                   -\frac{d_i+d_o}{J_i+J_o}& -\frac{d_i+d_o}{J_i+J_o} \\
                   -\frac{d_i+d_o}{J_i+J_o} & -\frac{d_i+d_o}{J_i+J_o} \\
                 \end{array}\right],
\end{split}
\end{equation*}
\begin{equation*}
\begin{array}{cc}
B^*=&(1-\mu)\left[
               \begin{array}{cc}
                \frac{1}{J_i}& 0 \\
                   0 & \frac{-1}{J_o} \\
               \end{array}
             \right]\end{array}\\
+\mu\left[\begin{array}{cc}
                 \frac{1}{J_i+J_o} & \frac{-1}{J_i+J_o} \\
                 \frac{1}{J_i+J_o} & \frac{-1}{J_i+J_o} \\
                 \end{array}\right].
\end{equation*}
with $\mu = (t-t_0)/T$, where $T$ is the period of the transient process. This leads to $\Sigma^* = \Sigma_1$ when $t=t_0$ ($\mu=0$) and $\Sigma^*=\Sigma_2$ when $t = T-t_0$ ($\mu=1$), respectively.

To do the simulation, we choose $J_i=0.2kgm^2$, $J_o=0.7753kgm^2$, $d_i=0.03Nms$, $d_o=0.03Nms$, and $T=0.86$. The initial time of the transient dynamics is denoted as $t_0=0$, and the terminal time is denoted as $t_1=1$. Let $(\omega_i(t_0),\omega_o(t_0))=(150,0)$  $\omega_i(t_1)=\omega_o(t_1)=25$, which means the clutch is engaged at $t_1$. We design a control law such that the trajectory of the closed-loop system composed of (\ref{e25.8}) and the control law starting from $(\omega_i(t_0),\omega_o(t_0))$ can reach approximately $(\omega_i(t_1),\omega_o(t_1))$ when $t=t_1$. Simulation result is shown in Fig. \ref{e2r1} for the time response of the closed-loop system.
\begin{figure}[!htb]
  \centering
  \includegraphics[width=9.5cm,height=5.5cm]{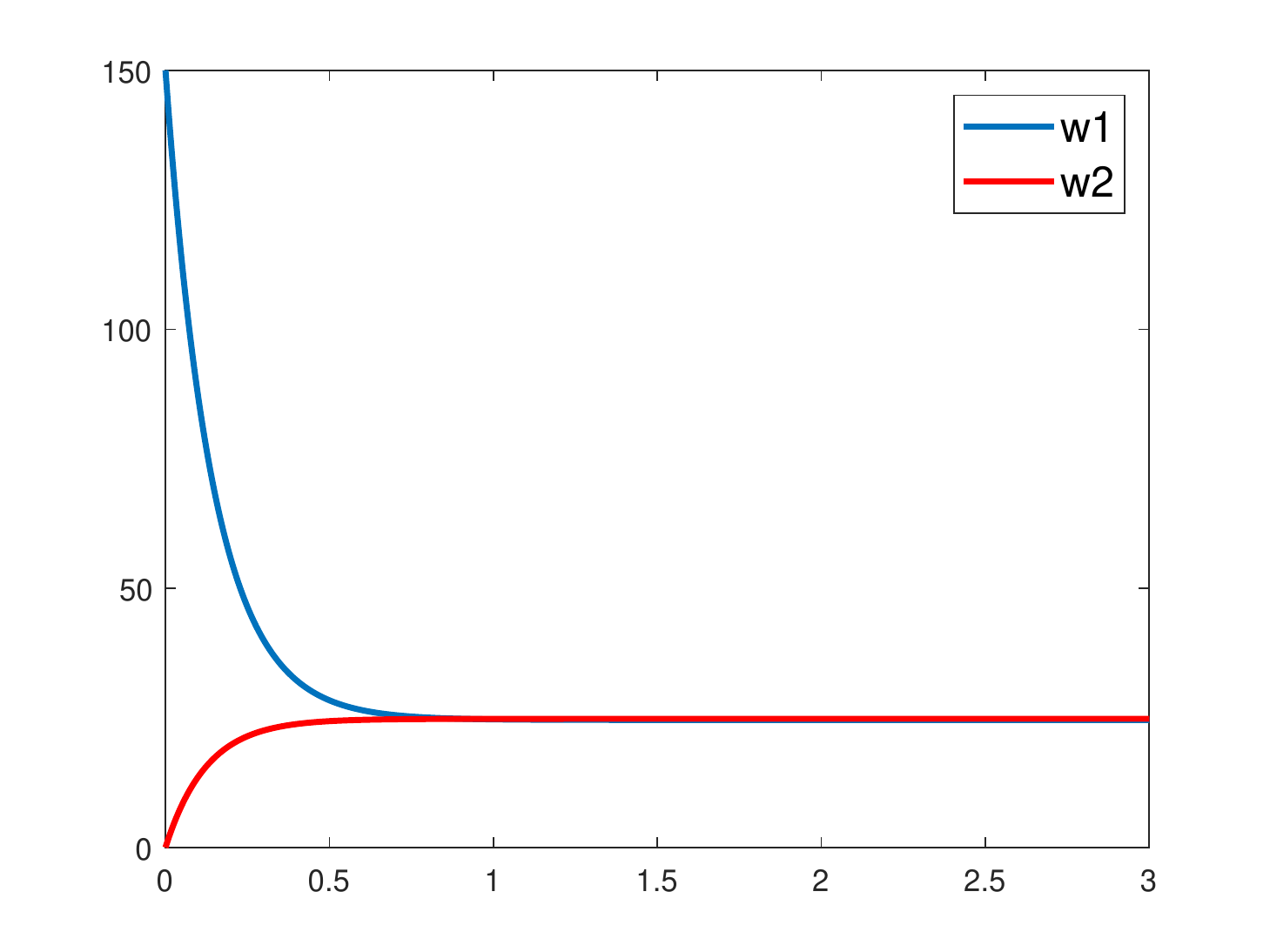}
  \caption{ The profile of states of the system (\ref{e25.8}) }\label{e2r1}
\end{figure}

\end{exa}

\section{Conclusion}

The problem of modeling  dimension-varying dynamic process of linear systems is investigated. First, the Euclidian spaces of various dimensions are put together to form a state space of dimension free systems. An cross dimensional addition is introduced, which provides a pseudo vector space structure on this dimension free state space. The inner product is then introduced, which suggests norm and distance. The metric topology follows, which makes the dimension free state space a path-wise connected topological space. The projection of vectors and then linear systems on different dimensional spaces are proposed. Set of matrices with different dimensions is considered as the general linear mappings on dimension free state space. Semi-tensor product is introduced on set of matrices, which turns the set into a semi-group. Finally, the S-system is obtained as the action of the semi-group of dimension-varying matrices on the dimension free vector space.

To make a trajectory ``cross" different dimensional Euclidian spaces an equivalence relation is proposed, which is basically deduced from the distance. Then the quotient space is obtained, which is a vector, metric, and Hausdorff space. A dimension-varying system can be properly projected on this quotient space, and a dynamic system on quotient space can be lifted to to Euclidean space of various dimensions. This project-lift process yields a technique to model dynamics of dimension-varying process.

Two examples are presented to demonstrate the design technique. One is a numerical example, which shows (to be completed). The other one is an engineering application. It demonstrated the control design technique for dimension-varying process of clutch system. A comparison with traditional method is also presented.

There are several interesting and challenging problems remain for further investigation. Some of them are as follows:
\begin{itemize}
\item[(i)]~~ What is the relationship of a linear (control) system with its projected system? Do they share some common properties?
\item[(ii)]~~ What is the practically meaningful model of the dynamics of dimensional process? To make the dynamics linear we propose to use a linear combination of the pre and after dynamic models. Is this approximation reasonable?
\item[(iii)]~~How to model large scale dimension-varying systems, such as internet?
\item[(iv)]~~ How to extend this approach to nonlinear case?
\end{itemize}

\end{document}